\documentclass[reqno]{amsart}
\usepackage{amsmath,amssymb,latexsym}
\usepackage[T1]{fontenc}
\usepackage{dsfont}
\usepackage{enumerate}

\usepackage{enumerate,xspace}

\usepackage{soul}

\usepackage{mdwlist}

\usepackage{xcolor}

\usepackage{nicefrac}

\usepackage{etoolbox} 

\usepackage[pdftex]{hyperref}

\theoremstyle{plain}
\newtheorem*{theoremA}{Theorem A}
\newtheorem*{theoremB}{Theorem B}
\newtheorem*{theoremC}{Theorem C}
\newtheorem*{theoremD}{Theorem D}
\newtheorem{theorem}{Theorem}[section]
\newtheorem{lemma}[theorem]{Lemma}
\newtheorem{proposition}[theorem]{Proposition}
\newtheorem{corollary}[theorem]{Corollary}
\newtheorem{fact}[theorem]{Fact}

\newcounter{proofcount}
\AtBeginEnvironment{proof}{\stepcounter{proofcount}}
\newtheorem{claim}{Claim}
\newtheorem*{claim*}{Claim}
\makeatletter                  
\@addtoreset{claim}{proofcount}
\makeatother                   

\newenvironment{claimproof}[1][Proof of Claim \theclaim.] 
{%
	\proof[#1]%
	
}
{%
	\endproof%
}

\newenvironment{claimproof*}[1][Proof of the Claim.] 
{%
	\proof[#1]%
	
}
{%
	\endproof%
}

\theoremstyle{definition}
\newtheorem{definition}[theorem]{Definition}
\newtheorem{remark}[theorem]{Remark}
\newtheorem{example}[theorem]{Example}
\newtheorem{question}{Question}

\newcommand{\dcl}{\textnormal{dcl}}
\newcommand{\acl}{\textnormal{acl}}
\newcommand{\acleq}{\textnormal{acl}^{\textnormal{eq}}}

\newcommand{\tp}{\textnormal{tp}}
\newcommand{\stp}{\textnormal{stp}}
\newcommand{\U}{\textnormal{U}}
\newcommand{\p}{\mathbb{P}}
\newcommand{\Cb}{\textnormal{Cb}}
\newcommand{\Aut}{\textnormal{Aut}}
\newcommand{\RM}{\textnormal{RM}}

\newtheorem*{akn}{Acknowledgments}

\def\Ind#1#2{#1\setbox0=\hbox{$#1x$}\kern\wd0\hbox to 0pt{\hss$#1\mid$\hss}
	\lower.9\ht0\hbox to 0pt{\hss$#1\smile$\hss}\kern\wd0}

\def\ind{\mathop{\mathpalette\Ind{}}}

\def\notind#1#2{#1\setbox0=\hbox{$#1x$}\kern\wd0
	\hbox to 0pt{\mathchardef\nn=12854\hss$#1\nn$\kern1.4\wd0\hss}
	\hbox to 0pt{\hss$#1\mid$\hss}\lower.9\ht0 \hbox to 
	0pt{\hss$#1\smile$\hss}\kern\wd0}

\def\nind{\mathop{\mathpalette\notind{}}}

\title{Additive covers and the Canonical Base Property}
\author{Michael Loesch}
\date{\today}
\address{ \, Abteilung f\"ur Mathematische Logik, Mathematisches Institut,
  Albert-Ludwigs-Universit\"at Freiburg, Ernst-Zermelo-Stra\ss e 1, D-79104
  Freiburg, Germany}
\thanks{Research supported by the Deutsche Forschungsgemeinschaft (DFG, 
German Research Foundation) - Project number 2100310201 and 2100310301, 
part of the ANR-DFG program GeoMod as well as by the German Academic 
Exchange Service (DAAD) - Kurzstipendien für Doktoranden, 2019/20 
(57438025)}
\email{loesch@math.uni-freiburg.de}

\begin{document}

\begin{abstract}
	We give a new approach to the failure of the Canonical Base Property 
	(CBP) in the so far only known counterexample, produced by Hrushovski, 
	Palacin and Pillay. For this purpose, we will give an alternative 
	presentation of the counterexample as an additive covers of an 
	algebraically closed field. We isolate two fundamental weakenings of 
	the CBP, which already appeared in work of Chatzidakis,  and show that 
	they do not hold in the counterexample. In order to do so, a study of 
	imaginaries in additive covers is developed, for elimination of finite 
	imaginaries yields a connection to the CBP. As a by-product of the 
	presentation, we notice that no pure Galois-theoretic account of the 
	CBP can be provided. 
\end{abstract}

\maketitle

\section{Introduction}
Internality is a fundamental notion in 
geometric model theory in order to 
understand a complete stable theory of finite Lascar rank in terms of its 
building blocks, its minimal types of rank one. 
A type $p$ is internal, resp. almost internal to 
the family $\p$ of all non-locally modular minimal types, if there exists 
a set of parameters $C$ such that every realization $a$ of $p$ is 
definable, resp. algebraic over $C,e$
where $e$ is a tuple of realizations of types (each one based over $C$) in $\p$.

Motivated by results of Campana \cite{fC80} on algebraic coreductions, 
Pillay and 
Ziegler \cite{PZ03} showed that in the finite rank part of the theory of 
differentially 
closed fields in characteristic zero, the type of the canonical base of a 
stationary type over a 
realization is almost 
internal to the constants. With this result, Pillay and Ziegler reproved 
the function field case of the Mordell-Lang conjecture in characteristic 
zero following Hrushovski's original proof but with considerable 
simplifications.

The above phenomena is captured in the notion of the Canonical Base 
Property 
(CBP), which was introduced and studied by Moosa and Pillay \cite{MP08}:
Over a realization of a 
stationary type, its canonical base is almost $\p$-internal.
Chatzidakis \cite{zC12} showed that the CBP already implies a seemingly 
stronger 
statement, the so-called uniform canonical base property (UCBP):
Whenever the type of a realization of the stationary type $p$ over 
some set $C$ of parameters is almost $\p$-internal, then so is
$\stp(\Cb(p)/C)$. 
For the proof, she isolated two remarkable properties which hold in 
every theory of finite rank with the CBP: Almost internality to $\p$ is preserved 
on intersections and more generally on quotients.
Motivated by her work, we introduce the following two notions. A 
stationary type is good, resp. special, if the condition for the CBP, 
resp. UCBP, holds 
for this type. (See Definitions \ref{D:good} and \ref{D:AB}
for a precise formulation.)
The following result relates these two notions to the aforementioned 
properties.

\begin{theoremA}\textup{(Propositions \ref{P:propA} 
	and \ref{P:propB})}
	The theory $T$ preserves internality on intersections, resp. on 
	quotients, if and only if 
	every stationary almost $\p$-internal type in  $T^\textnormal{eq}$ is 
	good, resp. special.
\end{theoremA}

Though most relevant examples of theories satisfy the CBP, Hrushovski, 
Palacín and Pillay \cite{HPP13} produced the so far only known example of 
an uncountably 
categorical 
theory without the CBP. We will give an alternative 
description of their counterexample in terms of additive covers of an 
algebraically closed field of characteristic zero. Covers are already 
present in early work of 
Hrushovski \cite{eH91}, Ahlbrandt and Ziegler \cite{AZ91} as well as of 
Hodges and Pillay \cite{HP94}. 
For an additive cover $\mathcal{M}$ of an algebraically closed field, the sort $S$ is the home-sort and
$P$ is the field-sort. The automorphism group $\Aut(\mathcal{M}/P)$ embeds 
canonically in 
the group of all additive maps on $P$. If the sort $S$ is almost $P$-internal, the CBP trivially holds.
The counterexample to the CBP has a ring structure on the sort $S$ and the ring multiplication
$\otimes$ is a lifting of the field multiplication. The automorphisms group over $P$ corresponds to the group of 
derivations, which ensures that the sort $S$ is not almost $P$-internal.
We prove the following 
result.

\begin{theoremB}\textup{(Propositions \ref{P:M1AutVersion} 
		and \ref{P:CBPpureCover})}
		The CBP holds whenever every 
		additive map on $P$ induces an automorphism in 
		$\Aut(\mathcal{M}/P)$. 
		If $\Aut(\mathcal{M}/P)$ corresponds to the group of 
		derivations, then the product $\otimes$
		is definable in $\mathcal{M}$.
\end{theoremB}

We focus on additive covers in which the sort $S$ is not almost 
$P$-internal, since otherwise the CBP trivially holds and show that no 
such additive cover can eliminate imaginaries. On the other side,
the counterexample to the CBP does eliminate finite imaginaries, which 
fits into situation:

\begin{theoremC}\textup{(Theorem \ref{T:finImagCBP})}
	If $\mathcal{M}$ eliminates finite imaginaries, 
	then it cannot preserve internality on quotients, so in particular 
	the CBP  does not hold.
\end{theoremC}

A standard argument shows that the CBP holds whenever it holds for all 
real stationary types. We will note that in the counterexample to the CBP 
the corresponding real versions of goodness and specialness hold, namely, 
every real stationary almost $P$-internal type is special. However the  
version for real types does not imply the full condition and gives a new 
proof of the failure of the CBP.

\begin{theoremD}\textup{(Propositions \ref{P: M1PropB} 
		and \ref{P: M1PropA})}
	The counterexample to the CBP does not preserve internality on intersections.
\end{theoremD}

Palac\'in and Pillay ~\cite{PP17} considered a strengthening of the CBP, 
called the strong canonical base property, which we show cannot 
hold in any additive cover, where $S$ is not almost $P$-internal.
Regarding a question which arose in \cite{PP17},
we prove that no \textit{pure} Galois-theoretic  account of the CBP can be 
provided.

In a forthcoming work, we use the approach with additive covers in order 
to produce new counterexamples to the CBP.

\begin{akn}
The author would like to thank his supervisor Amador Mart\'in Pizarro for 
numerous helpful discussions, his support, generosity and guidance.
He also would like to thank Daniel Palac\'in for multiple interesting 
discussions. Part of this research was carried out at the University of 
Notre Dame (Indiana, USA) with financial support from the 
DAAD, which the author gratefully acknowledges. The author would like to 
thank for the 
hospitality and Anand Pillay for many helpful discussions and for 
suggesting the study of imaginaries in the counterexample to the CBP.

\end{akn}	
	
\section{The Canonical Base Property and Related Properties}\label{S:CBPAB}

In this section we introduce two properties related to the canonical base 
property. We assume throughout this article a 
solid knowledge in geometric stability theory  \cite{aP96,TZ12}. 
Most of the results in this section can be found in \cite{zC12}. 

Let us fix a complete stable theory of finite Lascar rank. As usual, we 
work inside a sufficiently saturated ambient model. We denote by $\p$ the 
$\emptyset$-invariant family of all non-locally modular minimal types. 

The following notions provide an equivalent formulation of the CBP and the 
UCBP. They will play a crucial role in our attempt to weaken the CBP to 
other contexts.   

\begin{definition}\label{D:good}
 A stationary type $p$ is:
\begin{itemize}
	\item \emph{good} if $\stp(\Cb(p)/a)$ is almost 
	$\p$-internal for some (any) realization $a$ of $p$,
	\item \emph{special} if, for every parameter set $C$ and every 
	realization $a$ of $p$, whenever $\stp(a/C)$ almost $\p$-internal, so 
	is $\stp(\Cb(p)/C)$ almost $\p$-internal. 
\end{itemize}
\end{definition}

\begin{remark}\label{R:CBP_AB}~
	\begin{enumerate}[(a)]
		\item Note that every special type is good, by setting $C=\{a\}$. 
		\item It is 
		immediate from the definitions that the theory $T$ has 
		the CBP, resp.\  the UCBP, if and only if every 
		stationary type in  $T^\textnormal{eq}$ is good, resp.\ special. 
		\item Analog to \cite[Remark 2.6]{aP95}, it can be easily shown 
		that 
		whether or not every stationary type is good, resp. special, is 
		preserved 
		under naming parameters.
	\end{enumerate}
\end{remark}

Chatzidakis showed in \cite[Theorem 2.5]{zC12} that the CBP already 
implies the UCBP for (simple) theories of finite rank. In order to prove 
so, she first shows in \cite[Proposition 2.1]{zC12} that, under the CBP, 
the 
type 
$\tp(b/\acleq(a)\cap\acleq(b))$ is almost $\p$-internal, whenever 
$\stp(b/a)$ is almost $\p$-internal, and secondly in \cite[Lemma 
2.3]{zC12}, 
that  
$\tp(b/\acleq(a_1)\cap \acleq(a_2))$ is almost $\p$-internal, if both 
$\stp(b/a_1)$ and $\tp(b/a_2)$ are. Motivated by her work, we now 
introduce two notions capturing these intermediate steps and 
study their relation to the CBP. 

\begin{definition}\label{D:AB}
	The theory $T$ \emph{preserves internality on intersections} if 
	the type 
	\[\tp(b/\acleq(a)\cap\acleq(b))\]
	is almost $\p$-internal, 
	whenever $\stp(b/a)$ is almost $\p$-internal.
	Similarly, the theory 
	 \emph{preserves internality on quotients} if  the type 
	 \[\tp(b/\acleq(a_1)\cap \acleq(a_2))\] 
	 is almost $\p$-internal, 
	 whenever both $\stp(b/a_1)$ and $\tp(b/a_2)$ are.
\end{definition}

In order to relate the above properties to consequences of the CBP, we 
will need the following observation.

\begin{fact}\label{F:level}\textup{(}\cite[Proposition 1.18]{zC12} 
\textnormal{ \& } \cite[Theorem 3.6]{PW13}\textup{)}

Let $\stp(b/A)$  and $\stp(b/C)$ be two $\p$-analysable types. 
	\begin{enumerate}[(a)] 
		\item The type  $\stp(b/\acleq(A)\cap\acleq(C))$ is again 
		$\p$-analysable. In particular, so is  
		$\stp(b/\acleq(A)\cap\acleq(b))$ also $\p$-analysable.

	\item Let $b_A$ be the maximal 
	subset of $\acleq(A,b)$ such that $stp(b_A /A)$ 
	is almost $\p$-internal. The tuple $b_A$ (in some fixed enumeration) 
	dominates $b$ over 
	$A$, that is, for every set of parameters $D\supset A$,  
		\[ b \ind_A D \ \ \text{ whenever } \  \ b_A \ind_A D.\]
	Furthermore, whenever $\acleq(D)\cap\acleq(A,b_A)=\acleq(A)$, so is
	\[ \acleq(D)\cap\acleq(A,b)=\acleq(A).\]
	\end{enumerate}
\end{fact}

\begin{proposition}\label{P:propA}
	The theory $T$ preserves internality on intersections if and only if 
every stationary almost $\p$-internal type in  $T^\textnormal{eq}$ is good.
\end{proposition}
\begin{proof}
	
	We assume first that every stationary almost $\p$-internal type is 
	good, but 
	the conclusion fails, 
	witnessed by 
	two tuples $a$ and $b$. By Remark \ref{R:CBP_AB}, we may assume 
	\[ \acleq(a)\cap\acleq(b)=\acleq(\emptyset).\]

	Thus, the type $\stp(b/a)$ is 
	almost $\p$-internal, but the type $\stp(b)$ is not. Note that 
	$\stp(b)$ is $\p$-analysable, by Fact \ref{F:level}.

	Among all possible (imaginary) tuples in the ambient model take now 
	$a'$ such that 
	$\stp(b/a')$ is almost $\p$-internal and 
	\[\acleq(a')\cap\acleq(b)=\acleq(\emptyset)\]
	with $\U(b_\emptyset/a')$ maximal. Since $\stp(b/a')$ is almost 
	$\p$-internal, there is a set of parameters $A$ containing $a'$ with 
	$A \ind_{a'} b$ such that $b$ is algebraic over $Ae$, where $e$ is a 
	tuple of 
	realizations of types (each one based over $A$) in $\p$. Since each 
	type in the family $\p$ is minimal, we may assume, after possibly 
	enlarging $A$, that $e$ and $b$ are interalgebraic over $A$. 	
	
	 Let now $e'$ be a maximal 
	subtuple of $e$ independent from $b_\emptyset$ over $A$, so \[ e' 
	\ind_{A} 
	b_\emptyset \ \ \text{ and } \ \ e \in \acleq(A, e', b_\emptyset).\] 
	Hence, 
	the tuple $b$ 
	is algebraic over 
	$Ae' b_\emptyset $ and 
	\[\acleq(A,e')\cap\acleq(b_\emptyset)\subset 
	\acleq(a')\cap\acleq(b)=\acleq(\emptyset).\]
	Therefore 
	$\acleq(A,e')\cap\acleq(b) =\acleq(\emptyset)$, by Fact \ref{F:level}.
	
	Notice that $\stp(b/A, e')$ is almost $\p$-internal, yet this does not 
	yield any contradiction since 
	$\U(b_\emptyset/A,e')=\U(b_\emptyset/a')$.  Choose now $b'$ realizing 
	$\stp(b/A,e')$  independent from $b$ over $A, e'$. An easy forking 
	computation yields 
	\[ \acleq(b')\cap\acleq(b)=\acleq(\emptyset).\] By the hypothesis we 
	have that the almost $\p$-internal type 
	\[\stp(b'/\acleq(A,e'))=\stp(b/\acleq(A,e')) \] is good,
    so we deduce that $\stp(\Cb(b/A,e')/b')$ is almost $\p$-internal. 
    Remark that $b$ is algebraic over $\Cb(b/A,e', b_\emptyset)$ and thus 
    also algebraic over $b_\emptyset \Cb(b/A,e')$. 
    
    Putting all of the above together, we conclude that the type     
	$\stp(b/b')$ is almost $\p$-internal. Since 
	\[\U(b_\emptyset/b')\geq \U(b_\emptyset/A, e',b') = 
	\U(b_\emptyset/A,e')=\U(b_\emptyset/a'),\] 
	we deduce by the maximality of $\U(b_\emptyset/a')$ that 
	$\U(b_\emptyset/b')= \U(b_\emptyset/A, e',b')$, that is, \[ 
	b_\emptyset \ind_{\acleq(A,e')\cap \ \acleq(b')} A, e', b'.\]

	Hence 	$b_\emptyset \ind b'$, so 
	$b \ind b'$, by Fact \ref{F:level}, contradicting that $\stp(b)$ is 
	not almost $\p$-internal.
	
	For the other direction, we need to show that  the almost 
	$\p$-internal 
	type $\stp(a/b)$ is good, that is, that $\stp(\Cb(a/b)/a)$ is almost 
	$\p$-internal. We may assume that $b$ equals the canonical base 
	$\Cb(a/b)$. Superstability yields that $b$ is contained in the 
	algebraic closure of 
	finitely many $b$-conjugates of $a$. By preservation of internality on 
	intersections, the type  
	$\tp(a/\acleq(a)\cap\acleq(b))$ is almost $\p$-internal, so it follows 
	that \[\tp(b/\acleq(a)\cap\acleq(b))\] is almost $\p$-internal. Hence, 
	the type  $\stp(b/a)$ is almost $\p$-internal, as desired. 
~\end{proof}

It follows now from Remark \ref{R:CBP_AB} that preservation of internality 
on intersections does not depend on constants being named. 
\begin{corollary}\label{C:NamingParametersIntersections}
Preservation of internality on intersections is invariant under naming and 
forgetting parameters.
\end{corollary}

\begin{remark}\label{R:CCBP}
 It follows from Remark \ref{R:CBP_AB} and Proposition \ref{P:propA} that 
 the CBP is equivalent to the property that 
	whenever $b=\Cb(a/b)$, then $\tp(b/\acleq(a)\cap\acleq(b))$ is almost 
	$\p$-internal, which was already shown in \cite[Theorem 2.1]{zC12}.
\end{remark}

\begin{proposition}\label{P:propB}
	The theory $T$  preserves internality on quotients if and only if 
	every stationary almost $\p$-internal type in  $T^\textnormal{eq}$ is 
	special.
\end{proposition}

\begin{proof}
	
		Assume that every stationary almost $\p$-internal type is 
		special. We want 
		to show that 
		\[\tp(b/\acleq(a_1)\cap\acleq(a_2))\]
		is almost $\p$-internal, 
		whenever both $\stp(b/a_1)$  and $\stp(b/a_2)$ are. 
		By Remark \ref{R:CBP_AB}, we may assume that
		\[ \acleq(a_1)\cap\acleq(a_2)=\acleq(\emptyset).\] 
		Note that the type $\stp(b)$ is $\p$-analysable, by Fact 
		\ref{F:level}, so 
		recall that $b_\emptyset$ is the maximal almost $\p$-internal 
		subset of 
		$\acleq(b)$. 
		As in the proof of Proposition \ref{P:propA} there is a set of 
		parameters $A_1$ containing 
		$a_1$ such that $A_1 \ind_{a_1} b$ and $b$ is 
		interalgebraic over $A_1$ 
		with some tuple $e$ of realizations of types (each one based over 
		$A_1$) in $\p$. Choosing a maximal subtuple $e'$   of $e$ with 
		$e' \ind_{A_1} b_\emptyset$, it follows that $b$ is algebraic over 
		$b_\emptyset A_1 e'$ and that
		\[ 
		\acleq(b_\emptyset)\cap\acleq(A_1,e')\subset \acleq(a_1).
		\]
		Hence 
		\begin{equation*}
		\tag{$\star$}	
		\acleq(b)\cap\acleq(A_1,e')\cap\acleq(a_2)=\acleq(\emptyset),
		\end{equation*}
		by Fact \ref{F:level}. Since the almost 
		$\p$-internal type
		$\stp(b / A_1 , e')$ is special, we have that
		\[\stp(\nicefrac{\Cb(b / A_1 , e')}{a_2})\]
		is almost $\p$-internal. Therefore 
		\[\stp(\nicefrac{\Cb(b / A_1 , e')}{\acleq(A_1, 
		e')\cap\acleq(a_2)})\]
		is almost $\p$-internal by Remark \ref{R:CBP_AB}.
		Since
		\[ b \ind_{\Cb(b / A_1 , e'),b_\emptyset} A_1, e' \]
		and $b$ is algebraic over $b_\emptyset A_1 e'$, the tuple $b$ is 
		algebraic over 
		$\Cb(b / A_1 , e') b_\emptyset$. In particular, the type 
		\[ \stp(b/\acleq(A_1, e')\cap\acleq(a_2))  \]
		is almost $\p$-internal and hence so is $\stp(b)$ because of 
		($\star$).
	
	In order to prove the other direction, we want to show that the almost 
	$\p$-internal type $\stp(a/b)$ is special. Fix $C$ some a set of 
	parameters such that $\stp(a/C)$ is almost $\p$-internal. By  
	preservation of internality on quotients, the type 
	\[\stp(a/\acleq(b)\cap\acleq(C)) \]
	is almost $\p$-internal and so is 
	\[\stp(\nicefrac{\Cb(a/b)}{\acleq(b)\cap\acleq(C)}),\] 
	since the canonical base $\Cb(a/b)$ is algebraic over finitely many 
	$b$-conjugates of $a$. 
~\end{proof}
We deduce now the analog of Corollary \ref{C:NamingParametersIntersections}
for preservation of internality on quotients. 
\begin{corollary}\label{C:NamingParametersQuotients}
Preservation of internality on quotients is invariant under naming and 
forgetting parameters.
\end{corollary}

Thanks to the previous notions, we will provide for the sake of 
completeness a 
compact proof in Corollary \ref{C:Zoe} that the CBP already implies the 
UCBP, which essentially 
follows the lines of Chatzidakis's proof \cite[Theorem 2.5]{zC12}: Under 
the assumption of the CBP, the UCBP is equivalent to preservation of 
internality of quotients. Hence, we need only show in Proposition 
\ref{P:CBPQuotients} that the CBP implies the latter (cf. \cite[Lemma 
2.3]{zC12}). For this, we need some auxiliary results. 

Let $\Sigma$ denote the family of all minimal types, that is, of Lascar 
rank one. For a set $A$  of parameters, denote by $A_{\emptyset}^{\Sigma}$ 
be the maximal almost $\Sigma$-internal subset (in some fixed 
enumeration) of $\acleq(A)$.  

\begin{fact}\label{F:Comp}\textup{(}\cite[Lemma 1.10]{zC12}
\textnormal{ \& } \cite[Observation 1.2]{zC12}\textup{)}
	Assume that the types $\stp(e)$ and $\stp(c)$ are almost 
	$\Sigma$-internal.
	\begin{enumerate}[(a)] 
		\item If the tuple $e$ is algebraic over $Ac$ for some parameter 
		set $A$, 
		then $e$ is algebraic over 
		$A_{\emptyset}^{\Sigma} c$.
		\item If the type $\stp(c)$ is $\p$-analysable, then it is almost 
		$\p$-internal.
	\end{enumerate}	
\end{fact}

\begin{lemma}\label{L:CompCBP}
	Assume that the theory $T$ has the CBP and let $e$ be a tuple which is 
	algebraic over $AB$ with
	$\acleq(A)\cap\acleq(B)=\acleq(\emptyset)$. If the type $\stp(e)$ is 
	almost $\Sigma$-internal, then $e$ is  algebraic over 
	$A_{\emptyset}^{\Sigma} B$.
\end{lemma}
\begin{proof}
	Choose a set of parameters $D$ with $D \ind e,A,B$ such that $e$ is 
	interalgebraic over $D$ with a tuple of realizations of types (each 
	one based over $D$) in $\Sigma$.
	Since 
	\[ e \ind_{A_{\emptyset}^{\Sigma} B} D  \ \ \textnormal{ and } \ \ 
	\acleq(A,D)\cap\acleq(B,D)=\acleq(D),\]
	we may assume, after naming $D$, that $e$ is a single  
	element of Lascar rank one. If 
	\[ e \nind_{B} A_{\emptyset}^{\Sigma}, \]
	we are done. Otherwise
	\[ e \ind_{B} A_{\emptyset}^{\Sigma}, \]
	so \[ 
	\acleq(A_{\emptyset}^{\Sigma})\cap\acleq(B,e)=\acleq(A_{\emptyset}^{\Sigma})\cap\acleq(B)=
	 \acleq(\emptyset).\]
	The variant of Fact \ref{F:level} (b) with respect to $\Sigma$ yields
	\[\acleq(A)\cap\acleq(B,e)=\acleq(\emptyset).\]
	Now the CBP and Remark \ref{R:CCBP} imply that the type 
	$\stp(\Cb(B,e/A))$ 
	is almost 
	$\p$-internal, hence almost $\Sigma$-internal. Therefore, the 
	canonical base $\Cb(B,e/A)$ is contained 
	in $A_{\emptyset}^{\Sigma}$. Since $e$ 
	is algebraic over $\Cb(B,e/A) B$, we conclude that $e$ is 
	algebraic 
	over  
	$A_{\emptyset}^{\Sigma} B$, as desired.
~\end{proof}

We have now the necessary ingredients to show that every complete stable 
theory of 
finite rank with the CBP preserves internality on quotients.

\begin{proposition}\label{P:CBPQuotients}
	If the theory $T$ has the CBP, then it preserves internality on 
	quotients.
\end{proposition}
\begin{proof}
	We want to show that 
	\[\tp(b/\acleq(a_1)\cap\acleq(a_2))\]
	is almost $\p$-internal, whenever both $\stp(b/a_1)$  and 
	$\stp(b/a_2)$ are. Since the CBP is preserved under naming parameters, 
	we may assume that 
	\[\acleq(a_1)\cap\acleq(a_2)=\acleq(\emptyset).\]
	Choose sets of parameters $A_1$ containing $a_1$ and $A_2$ containing 
	$a_2$ 
	with
	\[  A_1 \ind_{a_1} b, a_2 \ \  \textnormal{ and  } \ \ A_2 \ind_{a_2} 
	b, A_1 \]
	such that $b$ is algebraic over both $A_1 e_1$ and $A_2 e_2$, where 
	$e_1$ and $e_2$ are tuples of realizations of types (each one based 
	over $A_1$, resp. $A_2$) in $\p$. Since
	\[\acleq(A_1)\cap\acleq(A_2)=\acleq(a_1)\cap\acleq(a_2)=\acleq(\emptyset),\]
	the CBP and Remark \ref{R:CCBP} implies that $\stp(\Cb(A_1/A_2))$ is 
	almost $\p$-internal, so 
	\[A_1 \ind_{(A_2)_{\emptyset}^{\Sigma}} A_2.\] 
	Choose now a maximal subtuple 
	$e_{1}'$ of $e_1$ which is independent from $A_2$ over $A_1$, so $e_1$ 
	is algebraic over $A_1 e_{1}' A_2$ and 
	\[\acleq(A_1 , 
	e_{1}')\cap\acleq(A_2)=\acleq(A_1)\cap\acleq(A_2)=\acleq(\emptyset).\]
	Now, let 
	$e_{2}'$ be a maximal subtuple of $e_2$ with
	\[ e_{2}' \ind_{A_2} A_1 , e_{1}'. \]
	We deduce that 
	\[ A_1 , e_{1}' \ind_{(A_2)_{\emptyset}^{\Sigma}} A_2 , e_{2}' \]
	and $e_2$ is algebraic over $A_1 e_{1}' e_{2}' A_2$. 
	Moreover 
	\[\acleq(A_1 , e_{1}')\cap\acleq(A_2, e_{2}')\subset\acleq(A_1, 
	e_{1}')\cap\acleq(A_2)=\acleq(\emptyset).\]
	By Lemma \ref{L:CompCBP}, 
	we get that $e_{1}$ is algebraic over 
	$(A_1,e_{1}')_{\emptyset}^{\Sigma} A_2$ and that 
	$e_{2}$ is algebraic over $A_1 e_{1}' (e_{2}' 
	A_2)_{\emptyset}^{\Sigma}$.
	It follows from Fact \ref{F:Comp} (a) that
	\[ (A_1,e_{1}')_{\emptyset}^{\Sigma} = (A_1)_{\emptyset}^{\Sigma} 
	e_{1}'\ \  \text{  and } \ \  (e_{2}',A_2)_{\emptyset}^{\Sigma} = 
	e_{2}' 
	(A_2)_{\emptyset}^{\Sigma}  .\] 
    We deduce  that $e_{1}$ is algebraic over $(A_1)_{\emptyset}^{\Sigma} 
	e_{1}' 
	A_2$ and $e_{2}$ is algebraic over $A_1 e_{1}' 
	e_{2}'(A_2)_{\emptyset}^{\Sigma}$. Therefore
	\[ A_1 , e_{1} \ind_{(A_1)_{\emptyset}^{\Sigma}, 
	(A_2)_{\emptyset}^{\Sigma}, e_{1}, e_{2}} 
	A_2 , e_{2}. \]
	Hence $b$ is algebraic over $(A_1)_{\emptyset}^{\Sigma}, 
	(A_2)_{\emptyset}^{\Sigma}, e_{1}, e_{2}$, so the type $\stp(b)$ is 
	almost 
	$\Sigma$-internal. Since, by Fact \ref{F:level}, the type $\stp(b)$ is 
	$\p$-analysable, we conclude by Fact \ref{F:Comp} (b) that $\stp(b)$ 
	is almost $\p$-internal, 
	as desired.
~\end{proof}

\begin{remark}\label{R:IndepQuotients}
It is easy to see that a weakening of preservation of internality 
on quotients holds in every complete stable theory of finite rank, when 
the quotients are independent: If the types $\stp(b/a_1)$ and 
	$\stp(b/a_2)$ are almost $\p$-internal 
	and $a_1 \ind a_2$, then the type $\stp(b)$ is almost $\p$-internal.
\end{remark}

For completeness, we now restate Chatzidakis's proof 
\cite[Theorem 2.5]{zC12} 
that the CBP implies the UCBP using the aforementioned terminology.  

\begin{corollary}\label{C:Zoe}
The CBP and UCBP are equivalent properties for theories of finite rank. 
\end{corollary}
\begin{proof}
The UCBP  clearly implies the CBP, similar to the remark that every 
special type is good. 	
	
We assume now that the theory has the CBP. We need to show that every 
type 
$\stp(a/b)$ is special. Since
\[ \Cb(a/b)=\Cb(\nicefrac{\Cb(b/a)}{b}),\]
we may assume that $a$ is the 
 canonical base $\Cb(b/a)$. In particular, the type $\stp(a/b)$ is almost 
 $\p$-internal, by the CBP. Now, the Propositions \ref{P:CBPQuotients} and 
 \ref{P:propB} yield that the type $\stp(a/b)$ is special, as desired.
~\end{proof}

The equivalence of the previous corollary  motivates the following 
question, after localizing to almost $\p$-internal types.

\begin{question}\label{Q:1}
	Are  preservation of internality on intersections and on quotients 
	equivalent 
	properties for theories of finite rank? 
\end{question}

At the moment of writing, we do not know whether the previous question has 
a positive answer. Note that providing a structure which answers 
negatively the above question means in particular a new theory of finite 
rank without the CBP, since we will see in Section \ref{S:Imag} that the 
so far only known counterexample to the CBP given in \cite{HPP13} does not 
preserve 
internality on intersections.

It was remarked in \cite[Lemma 2.11]{BMPW12} that the CBP holds whenever 
it holds for stationary real types, or equivalently, for real types over 
models. A natural question is whether the same holds for the above 
properties of preservation of internality.

\begin{question}\label{Q:2}
	Does a theory of finite rank preserve internality on intersections, 
	resp. on quotients, if every stationary real almost $\p$-internal type 
	is good, resp. special?
\end{question}
Additive covers of the algebraically closed field $\mathbb{C}$, which will 
be introduced in the following section, will provide a negative answer 
(see Corollary \ref{C:AB_imag}) to Question \ref{Q:2}. 
	
\section{Additive Covers}\label{S:AddCovers}

The only known example so far of a stable theory of finite rank without 
the CBP appeared in \cite{HPP13}. We will consider this example from the 
perspective of additive covers of the algebraically closed field 
$\mathbb{C}$. We start this section with a couple of definitions. 

Following the terminology of Hrushovski ~\cite{eH91}, Ahlbrandt and 
Ziegler ~\cite{AZ91}, and 
Hodges and Pillay ~\cite{HP94}, we say that $M$ is a \textit{cover} of $N$ 
if the following three conditions hold:
\begin{itemize}
	\item The set  $N$ is a stably embedded $\emptyset$-definable subset 
	of $M$.
	\item There is a surjective $\emptyset$-definable map $\pi:M\backslash 
	N\rightarrow N$. 
	\item There is a family of groups $(G_a)_{a\in N}$ definable in 
	$N^{\text{eq}}$ without parameters such that $G_a$ acts definably and 
	regularly on the fiber $\pi^{-1}(a)$. 
	\end{itemize}
For the purpose of this article, we will concentrate on particular covers 
of the algebraically closed 
field $\mathbb{C}$, and hence provide a  definition adapted to this 
context. From now on, given the canonical projection of the sort 
$S=\mathbb{C}\times\mathbb{C}$ onto the first coordinate $P=\mathbb{C}$, 
we will denote the elements of $P$ with the greek letters 
$\alpha$, $\beta$, ect., while the elements of $S$ will be seen
accordingly as 
pairs $(\alpha,a')$ and so on.
\begin{definition}\label{D:AdditiveCover}
An \emph{additive cover} of the algebraically closed 
field $\mathbb{C}$ is a structure 
$\mathcal{M}=(P,S,\pi,\star,\ldots)$ with the distinguished 
sorts $P=\mathbb{C}$ and $S=\mathbb{C}\times\mathbb{C}$ such that the 
following conditions to hold:
\begin{itemize}
	\item The structure $\mathcal{M}$ is a reduct of 
	$(\mathbb{C},\mathbb{C}\times\mathbb{C})$ with 
	the full field structure on the sort $P$. 
	\item The projection $\pi$ maps $S$ onto $P$. 
	\item There is an action $\star$ of $P$ on $S$ given 
	by $\alpha\star(\beta,b')=(\beta,b'+\alpha)$.
\end{itemize}
Moreover, the map
\[	\begin{array}{rccc}
	\oplus: & S\times S & \rightarrow & S \\[2mm]
	& \big((\alpha,a'),(\beta,b')\big)&\mapsto& (\alpha+\beta,a'+b')
	\end{array}\]
is definable in $\mathcal{M}$ without parameters.
\end{definition}

\begin{example}~\label{E:covers}
	
	\begin{itemize}
		\item The additive cover
		$\mathcal{M}_{0}=(P,S,\pi,\star,\oplus)$ with no additional 
		structure.
		\item The additive cover 
		$\mathcal{M}_{1}=(P,S,\pi,\star,\oplus,\otimes)$ 
		with the product
    \[	\begin{array}{rccc}
	\otimes: & S\times S & \rightarrow & S \\[2mm]
	& \big((\alpha,a'),(\beta,b')\big)&\mapsto& (\alpha\beta,\alpha 
	b'+\beta 
		a').
	\end{array}\]
	Note that $\mathcal{M}_{1}$ 
		is a commutative ring with multiplicative neutral element $(1,0)$. 
		The zero-divisors are exactly the elements $a$ in $S$ with 
		$\pi(a)=0$, that is, the pairs $a=(0, a')$.
	\end{itemize}

Given an additive cover $\mathcal{M}$, there is a 
canonical 
embedding
\[\begin{array}{rll}
\Aut(\mathcal{M}/P)& \hookrightarrow & \{F:\mathbb{C}\rightarrow\mathbb{C} 
\text{ 
	additive} 
\}\\[1mm]
\sigma &\mapsto & F_\sigma
\end{array}\]
uniquely determined by the identity
$\sigma(x)=F_\sigma(\pi(x))\star x$. 

For the additive cover $\mathcal M_0$ 
of Example \ref{E:covers}, the above embedding defines a bijection
\[\Aut(\mathcal{M}_{0}/P)\leftrightarrow\{F:\mathbb{C}\rightarrow\mathbb{C}
 \text{ additive}\}\] 
and a straight-forward calculation yields that 
		\[\Aut(\mathcal{M}_{1}/P)\leftrightarrow\{F:\mathbb{C}\rightarrow\mathbb{C}
		 \text{ 
		derivation}\}.\] Indeed, for elements $a=(\alpha,a')$ and 
		$b=(\beta,b')$ in $S$, we have
	\begin{align*}
	\sigma(a\otimes b)&=F_{\sigma}(\alpha\beta)\star (a\otimes b) 
	\textnormal{ and } \\
	\sigma(a)\otimes\sigma(b)&=\big(F_{\sigma}(\alpha)\star a\big)\otimes 
	\big(F_{\sigma}(\beta)\star b\big)=(\alpha\beta,\alpha 
	(b'+F_{\sigma}(\beta))+\beta (a'+F_{\sigma}(\alpha)))\\&=\big(\alpha 
	F_{\sigma}(\beta)+\beta F_{\sigma}(\alpha)\big)\star (a\otimes b).
	\end{align*}
\end{example}

\begin{remark}\label{R:genCov}
Every additive cover $\mathcal{M}$ is a saturated
uncountably categorical structure, where $P$ is the unique strongly 
minimal set up to non-orthogonality. The sort $S$ has Morley 
rank two and degree one, and is $P$-analysable in two steps. Moreover, 
each fiber 
$\pi^{-1}(\alpha)$ is strongly minimal.

Therefore, for additive covers, almost $\p$-internality in 
the CBP is 
equivalent to almost internality to $P$.  If $S$ is almost $P$-internal, 
then 
the CBP trivially holds. 
\end{remark}

\begin{remark}\label{R: Cex}
The counterexample to the CBP given in \cite{HPP13} is an 
additive cover, including for every irreducible variety $V$ defined over 
$\mathbb{Q}^{\textnormal{alg}}$ a predicate in the sort $S$ for the 
tangent bundle of $V$. It is easy to see that this structure has the 
same definable sets as the additive cover $\mathcal{M}_1$ given in Example 
\ref{E:covers}, since every polynomial expression over 
$\mathbb{Q}^{\textnormal{alg}}$ in $P$ lifts to a polynomial equation in 
$S$, using the ring operations $\oplus$ and $\otimes$. 

A key ingredient in the proof that the sort $S$ in the above counterexample
is not almost 
$P$-internal \cite[Corollary 3.3]{HPP13} is that every derivation on 
the algebraically closed field $\mathbb{C}$ induces an automorphism  
 in $\Aut(\mathcal{M}_1 / P)$. 
\end{remark}

For the following sections, we will need some auxiliary lemmas on the 
structure of  additive covers, and particularly those where 
the sort $S$ is not almost $P$-internal. For the sake of completeness, 
note that there are additive covers, besides the full structure, where the 
sort $S$ is $P$-internal: Consider the additive cover $\mathcal{M}$ with 
the following binary relation $R$ on $S\times S$
\[ R((\alpha,a'),(\beta,b')) \iff \big(\alpha\notin\mathbb{Q}\ \  \& \ \  
\beta\notin\mathbb{Q} \ \  \& \ \ a'=b' \big).
 \] 
It is easy to verify that $\Aut(\mathcal{M}/P)=(\mathbb{C},+)$ and the 
sort $S$ is $P$-algebraic (actually $P$-definable), after 
naming any element in the fiber $\pi^{-1}(1)$.

The following notion will be helpful in the following chapter. 

\begin{definition}\label{D:mean}
Given elements $a_1=(\alpha, a_1'),\ldots, a_n=(\alpha, a'_n)$ of $S$ all 
in the same fiber $\pi^{-1}(\alpha)$, their \emph{average} is the element 
\[\Big(\alpha, \frac{a'_1+\ldots+a'_n}{n}\Big).\]
\end{definition}

\begin{lemma}\label{L:mean}
	Given a non-empty finite set $A$ of elements of $S$, all lying in the 
	same fiber, every
	automorphism $\sigma$ of the additive cover maps the average of $A$ to 
	the average of $\sigma[A]$. In particular, the average of $A$ is 
	definable 
	over $A$.
\end{lemma}
\begin{proof}
    We proceed by induction on the size $n$ of the non-empty set $A$. For 
    $n=1$, there is nothing to prove. Assume $A$ contains at least two 
    elements, and choose  
	 $a$ some element of $A$. Set $b=\sigma(a)$. 
	Inductively, we have that $\sigma$ maps the average $d_1$ of 
	$A\backslash\{a\}$ 
	to the average $d_2$ of $\sigma[A]\backslash\{b\}$. Let 
	$\varepsilon_1$ and $\varepsilon_2$ be the unique elements in $P$ such 
	that $\varepsilon_1\star 
	d_1=a$ and $\varepsilon_2\star d_2=b$. A straight-forward  computation 
	yields that 
	$\frac{\varepsilon_1}{n}\star d_1$, resp. 
	$\frac{\varepsilon_2}{n}\star d_2$, is the average of $A$, resp. of 
	$\sigma[A]$. Now the 
	claim follows since $\sigma$ maps $\frac{\varepsilon_1}{n}$ to 
	$\frac{\varepsilon_2}{n}$.
~\end{proof}

\begin{lemma}\label{L:stat}
	Let $a_1=(\alpha_1,0),\ldots,a_n=(\alpha_n,0)$ be elements in $S$. The 
	type $\tp(a_1,\ldots,a_n / \alpha_1,\ldots,\alpha_n)$ is 
	stationary.	
\end{lemma}
\begin{proof}
	Choose a maximal subtuple $\hat{a}$  of $(a_1,\ldots,a_n)$ (algebraic) 
	independent 
	over the tuple $\bar{\alpha}=(\alpha_1,\ldots,\alpha_n)$. Note that 
	each
	$a_i$ is algebraic over $\bar{\alpha},\hat{a}$.  Set
	$b_i=(\alpha_i,b_{i}')$ the average of the finite set of 
	$\{\bar{\alpha},\hat{a}\}$-conjugates of $a_i$. The element $b_i$ is 
	definable over $\bar{\alpha},\hat{a}$, by Lemma \ref{L:mean}.
	
	\begin{claim*}
	The second coordinate $b_{i}'$ of the average $b_i$  is definable (as 
	an element of $P$) over $\bar{\alpha}$.
	\end{claim*}
	\begin{claimproof*}
	We need only show that $b_{i}'$ is fixed by every automorphism $\tau$ 
	of the sort $P$ fixing $\bar{\alpha}$. The map 
	$\sigma=(\tau,\tau\times \tau)$ is an automorphism of the full 
	structure $(\mathbb{C},\mathbb{C}\times\mathbb{C})$, and hence of the 
	reduct $\mathcal M$. Since $\tau(0)=0$, the automorphism $\sigma$ 
	fixes $\bar{\alpha},a_1,\ldots, a_n$. Hence $\sigma(b_i)=b_i$, so in 
	particular $\tau(b_{i}')=b_{i}'$.\end{claimproof*}
	Therefore $a_i=(-b_{i}')\star b_i$ is definable 
	over $\bar{\alpha},\hat{a}$. Since the fibers of the projection $\pi$ 
	are 
	strongly minimal (see Remark \ref{R:genCov}),  the type
	$\tp(\hat{a}/\bar{\alpha})$ is  stationary, so we obtain the desired 
	conclusion.
~\end{proof}
The above proof yields in particular the following: 
\begin{remark}\label{R:algFib}
	Every automorphism $\tau$ of $P$ fixing a subset $A$ induces an 
	automorphism $\sigma$ of the additive cover which fixes all the 
	elements of $S$ of 
	the form $(\alpha,0)$, with $\alpha$ in $A$. 
	
   The definable and algebraic closure of $P$ in the sort $S$ 
   coincide: 		\[ S \cap \acl(P) = S \cap 
   \dcl(P).  \]
	
	Given a set of parameters $B$ in $S$ and an element $\beta$ in the 
	sort $P$, all elements of the strongly minimal fiber $\pi^{-1}(\beta)$ 
	have 
	the same type over $B, P$ whenever the element $b=(\beta,0)$ of 
	$S$ is not algebraic over $B, P$. 
\end{remark}

\begin{lemma}\label{L:indep}
	Let $a_1=(\alpha_1,0),\ldots,a_n=(\alpha_n,0)$ be elements in the sort 
	$S$ with generic independent elements $\alpha_i$ in $P$.
	If the sort $S$ is not almost $P$-internal, then the $a_i$'s are 
	generic independent.
\end{lemma}
\begin{proof}
	Choose some $\beta$ generic in $P$ independent from 
	$\alpha_1$ and set $a=(\alpha_1,\beta)$ in $S$. Note that the Morley 
	rank of $a$ is two. If $a_1$ were not generic, then  
	$a_1$ 
	must be algebraic over the generic element $\alpha_1$ of $P$. Since 
	$a=\beta\star a_1$, it would 
	follow that the generic element 
	$a$ of $S$ is algebraic over $P$, which contradicts our assumption that
	the sort $S$ is not almost $\p$-internal. Hence $a_1$ is generic in 
	$S$.
	
	Now, we inductively assume that the tuple $(a_1,\ldots,a_{n-1})$ 
	consists of generic independent elements and want to show that $a_n 
	\ind \bar{a}_{<n}$. Assume otherwise that $a_n \nind 
	\bar{a}_{<n}$.
	Note that $\alpha_n$ is not algebraic over $\bar{a}_{<n}$, by 
	Remark \ref{R:algFib}, since $\alpha_n$ is not algebraic over   
	$\bar{\alpha}_{<n}$. Thus $a_n  
	\nind_{\alpha_n} \bar{a}_{<n}$, so $a_n$ is algebraic over 
	$\alpha_n \bar{a}_{<n}$.
	Choose now some  element $\gamma$ in $P$ generic over 
	$(\alpha_1,\ldots, \alpha_n)$ 
	and set $c=(\alpha_n,\gamma)=\gamma\star a_n$ in $S$. Note that $c$ is 
	algebraic over 
	$\bar{a}_{<n}P$. Observe that $\RM(c/\bar{a}_{<n})=2$, by the choice 
	of $\gamma$, so we conclude that $S$ is almost $\p$-internal, which 
	gives the desired contradiction. 
~\end{proof}

We conclude this section with a full description of the Galois groups of 
stationary 
$P$-internal types 
in additive covers, whenever the sort $S$ is not almost $P$-internal.

\begin{remark}\label{R:Gal}
	If $S$ is not $P$-internal, then every definable subgroup of 
	$(\mathbb{C}^n,+)$ appears as a Galois group and conversely every 
	Galois group is (definably isomorphic to) such a subgroup.
\end{remark}
\begin{proof}
	We show first that every definable subgroup $G$ of $(\mathbb{C}^n,+)$ 
	appears as a Galois group. Let 
	$a_1=(\alpha_1,0),\ldots,a_n=(\alpha_n,0)$ be elements in the sort 
	$S$ with generic independent elements $\alpha_i$ in $P$ and
	set
	\[ E = \{ \bar{x}\in S^n \ |\ \exists \bar{g}\in G \bigwedge_{i=1}^{n} 
	g_i \star a_i = x_i    \}.  \]
	The type $\stp(a_1,\ldots,a_n/\ulcorner E\urcorner)$ is $P$-internal 
	because $\alpha_1,\ldots,\alpha_n$ are definable over 
	$\ulcorner 
	E\urcorner$. We show that $G$ is the Galois group of this type.  
	If
	\[ b_1,\ldots,b_n \equiv_{\ulcorner E\urcorner, P} a_1,\ldots,a_n, \]
	then $\bar{b}$ is in $E$ and there is an unique $\bar{g}$ in 
	$G$ with $\bar{g}\star \bar{a}=\bar{b}$.
	Now assume that conversely $\bar{g}\star \bar{a}=\bar{b}$ for some 
	$\bar{g}$ in $G$. Note that for $1\leq k\leq n$ the element $a_k$ is 
	not algebraic over $\bar{a}_{<k},P$ by Lemma \ref{L:indep}, since $S$ 
	is not almost $P$-internal. Hence, Remark \ref{R:algFib} yields  
	that all elements in the fiber $\pi^{-1}(\alpha_k)$ have the same type 
	over $\bar{a}_{<k},P$. This shows that we can inductively construct an 
	automorphism $\sigma$ in $\Aut(\mathcal{M}/P)$ with $\sigma(a_k)=g_k 
	\star a_k$ for $1\leq k\leq n$.  The 
	automorphism $\sigma$ determines an element of the Galois group of the 
	fundamental type $\stp(a_1,\ldots,a_n/\ulcorner E\urcorner)$.
	
	Now we show that every Galois group is of the claimed form.
	The Galois group $G$ of a  
	real stationary fundamental 
	$P$-internal type 
	$\tp(a_1,\ldots,a_n/B)$
	equals 
	\[ G = \{ (g_1,\ldots,g_n)\in P^n \ \vert \  
	g_1 \star a_1 , \ldots , g_n \star a_n \equiv_{B,P} a_1 , 
	\ldots, a_n   \}.   \]
	More generally, given an imaginary element $e=f(a)$, where $a$ is a 
	real tuple and $f$ is an $\emptyset$-interpretable function,  the 
	Galois group of the stationary fundamental 
	$P$-internal type $\tp(e/B)$ equals 
	\[ \{ g\in P^n \ \vert \  
	f(g\star a) \equiv_{B,P} e   \}.   \]
	Hence, the statement follows, since every Galois group is the Galois 
	group of a stationary fundamental (possibly imaginary)  type.
	Note that we did not use here that $S$ is not almost $P$-internal.
~\end{proof}

\section{Imaginaries in additive covers}\label{S:Imag}

In order to answer Question \ref{Q:2}, we are led to study imaginaries in 
additive covers, with a particular focus to the additive covers 
in the Example \ref{E:covers}. We will 
first show that neither the counterexample $\mathcal{M}_1$ to the 
CBP of \cite{HPP13} nor the additive cover  $\mathcal{M}_0$ eliminate 
imaginaries. 

\begin{lemma}\label{L:imag}
	The additive cover $\mathcal{M}$ does not eliminate imaginaries if 
	every derivation on $\mathbb{C}$ induces an automorphism 
	in $\Aut(\mathcal{M}/P)$. 
\end{lemma}

\begin{proof}
	Choose two generic independent elements $\alpha$ and $\beta$  in 
	the sort $P$, and pick elements $a$ and $b$  in the fiber of $\alpha$ 
	and $\beta$, respectively.  Fix a derivation $D$ with kernel 
	$\mathbb{Q}^{\text{alg}}$.  Let us assume for a contradiction that the 
	definable set
	\[E=\{(x,y)\in S^2 \ |\ \exists (\lambda,\mu)\in P^2 (\lambda\star a=x 
	\ \& \  
	\mu\star b=y \ \& \ D(\beta)\lambda-D(\alpha)\mu=0)\}\] has a real 
	canonical parameter $e$. By hypothesis, the derivation $D$ induces an 
	automorphism $\sigma_{D}$ in $\Aut(\mathcal{M}/P)$. Note that 
	$\sigma_{D}$ must fix $E$ 
	setwise, because 
	$D(\beta)D(\alpha)-D(\alpha)D(\beta)=0$. Therefore (every element 
	of) the tuple $e$ lies in $P\cup\pi^{-1}(\mathbb{Q^{\text{alg}}})$. 
	In particular, the definable set $E$ is permuted by every automorphism 
	induced by a derivation. Now 
	let $D_1$ be a derivation with $D_1(\alpha)=1$ and $D_1(\beta)=0$, and 
	note that 
	$\sigma_{D_1}$ does not permute $E$,  since 
	$D(\beta)\cdot 1-D(\alpha)\cdot 0=D(\beta)\neq 0$, which gives the 
	desired contradiction.
~\end{proof}

The proof of \cite{HPP13} shows that the sort $S$ in an additive cover 
$\mathcal M$ is not almost $P$-internal, whenever every derivation on 
$\mathbb{C}$ induces an automorphism in $\Aut(\mathcal{M}/P)$. We will now 
give a strengthening of Lemma \ref{L:imag}.

\begin{proposition}\label{P:ImagInt}
	If the additive cover $\mathcal{M}$ eliminates imaginaries, then the 
	sort $S$ is $P$-internal.
\end{proposition}
\begin{proof}
 We will mimic the proof of Lemma \ref{L:imag}. Assume for a 
 contradiction that the sort $S$ is not $P$-internal and choose two 
 generic independent elements $a$ and $b$ in $S$. 	Since $S$ is not 
 $P$-internal, there is 
	an automorphism $\tau$ in $\Aut(\mathcal{M}/P)$ which fixes $b$ and 
	moves 
	$a$. If we can construct an automorphism $\sigma$ (which was 
	$\sigma_D$ in the proof of Lemma \ref{L:imag}) such that it only fixes 
	the 
	definable closure of $P$ (in $S$), we conclude as before that the real 
	canonical parameter $e$ of the definable set \[E=\{(x,y)\in S^2 \ 
	|\ 
	\exists (\lambda,\mu)\in P^2 (\lambda\star 
	a=x 
	\ \& \  
	\mu\star b=y \ \& \ F_\sigma(\beta)\lambda-F_\sigma(\alpha)\mu=0)\}\] 
 is definable over $P$. The automorphism $\tau$ fixes $e$, yet it maps the 
 pair 
 $(a,b)$ in $E$ outside of $E$. 

Hence, we need only show in the rest of the proof that there exists such 
an automorphism $\sigma$. 

Choose an enumeration  of elements 
$a_i=(\alpha_i,0)$ and $b_i=(\beta_i,0)$ in $S$, with $i<2^{\aleph_0}$, 
such that:
\begin{itemize}
	\item The tuple
	$\bar{\alpha}=(\alpha_i)_{i<2^{\aleph_0}}$
	is a transcendence basis of the algebraically closed field 
	$\mathbb{C}$.
	\item For each $i<2^{\aleph_0}$, the element $b_i$ is not algebraic 
	 over 
	$\bar{a},\bar{b}_{<i}$, where 
	$\bar{a}=(a_i)_{i<2^{\aleph_0}}$ and 
	$\bar{b}=(b_i)_{i<2^{\aleph_0}}$. Hence 
	$\RM(b_i/\bar{a},\bar{b}_{<i})=1$ since $\beta_i$ is in 
	$\acl(\bar{a})$.
		\item Each element in $S$ is algebraic over $\bar{a},\bar{b}$.
\end{itemize}
We denote by ${\langle \alpha \rangle}_i$ the unique subtuple of 
$\bar{\alpha}$ of smallest length such that $\beta_i$ is algebraic over 
${\langle \alpha \rangle}_i$. Write $\mathcal X$ for the set of all 
finite subtuples of $\bar{\alpha}$ and consider the map 
\[	\begin{array}{rccc}
\Phi: &  \mathcal X & \rightarrow & 
2^{\aleph_0} \\[2mm]
& (\alpha_{i_1},\ldots,\alpha_{i_n})&\mapsto& \max(i_1,\ldots,i_n).
\end{array}\]
The partial function $F$ defined by
\[ 
F(\alpha_i)=\alpha_{\omega^{i+1}}
 \ \ \textnormal{ and } \ \ 
F(\beta_i)=\alpha_{\omega^{\max\left(i,\Phi({\langle \alpha 
\rangle}_i)\right)+1}+\omega^{i}}
\]
is clearly injective. It follows inductively by Remark \ref{R:algFib}  
that 
\[\bar{a},\bar{b} \equiv_{P} 
F(\bar{a})\star\bar{a},F(\bar{b})\star\bar{b},\]
so $F$ induces a partial automorphism fixing $P$ pointwise with domain  
the set
\[\big(\pi^{-1}(\bar{\alpha})\times\pi^{-1}(\bar{\beta})\big)\cup P.\]
 
Given an element $c$ in $S$, it is by construction algebraic over $\bar a, 
\bar b$, so the average of its conjugates is definable over $\bar a, \bar 
b$, 
by Lemma \ref{L:mean}. Thus, every element  of $S$ is definable over 
$\bar{a}, \bar{b}, P$. Therefore the above partial automorphism extends 
uniquely to an automorphism $\sigma$ in $\Aut(\mathcal{M}/P)$.	
\begin{claim*}
The automorphism  $\sigma$ only fixes the definable 
closure of $P$ in $S$. 
\end{claim*}
\begin{claimproof*}
 Since $\sigma$ fixes the sort $P$, it suffices to show that all 
elements $c$ fixed by $\sigma$ of the form $c=(\gamma,0)$ are definable 
over $P$.  
Otherwise, choose subtuples of least possible length 
\[  \hat{a}=(a_{i_1},\ldots,a_{i_m}) \ \  \textnormal{ and } \ \ 
\hat{b}=(b_{j_1},\ldots,b_{j_n})
\]
of $\bar{a}$ and $\bar{b}$ such that $c$ is 
definable over  $\hat{a},\hat{b},P$. Note that $\max(n,m)>0$ and that 
every element in the fiber 
of 
$\gamma$ is definable over $\hat{a},\hat{b},P$.
The type
\[ 
p=\tp(\hat{a},\hat{b},c/\hat{\alpha},\hat{\beta},\gamma)
\]
is fundamental and stationary by Lemma \ref{L:stat}. Its Galois 
group $G$ is a definable additive subgroup of 
$\mathbb{C}^{m+n+1}$, by Remark \ref{R:Gal}.
If $\gamma$ is not algebraic over 
$\hat{\alpha},\hat{\beta}$, Lemma 
\ref{L:indep} yields that 
$c \ind \hat{a},\hat{b}$, so the type $\stp(c)$ is $P$-internal and 
hence so is (the generic element in the fiber $\pi^{-1}(\gamma)$ of) $S$,  
contradicting our assumption. 
Since  the Galois group $G$ of 
$p$ is definable over  $\{ 
\hat{\alpha},\hat{\beta},\gamma \}$, we deduce that it is definable over 
\[A=\acl(\hat{\alpha},
{\langle \alpha \rangle}_{j_1},\ldots,{\langle \alpha 
\rangle}_{j_n})\supset  \{ 
\hat{\alpha},\hat{\beta}\}.\] The group $G$ is given by a system 
$\mathcal{G}$ of 
linear equations of the form 
\begin{equation*}
	\lambda_{1} \cdot x_{1}+\dots+\lambda_{m+n+1} \cdot 
x_{m+n+1}=0,
\end{equation*}
with coefficients $\lambda_{i}$ in $A$ and the tuple
\[ ( F(\alpha_{i_1}),\ldots, F(\alpha_{i_m}), 
F(\beta_{j_1}),\ldots, F(\beta_{j_n}), 
0) \]
is a solution.
Set now $\gamma= \Phi\big( (\hat{\alpha},
{\langle \alpha \rangle}_{j_1},\ldots,{\langle \alpha 
	\rangle}_{j_n})\big) <2^{\aleph_0}$. If $\alpha_\gamma=\alpha_{i_k}$ 
	for some $1\le k\le m$, denote $i(\gamma)=i_k=\gamma$. Otherwise set 
	$i(\gamma)=j_\ell$  if $1\le \ell\le n$ is the least index such that 
	$\alpha_\gamma$ is an element in the tuple 
	${\langle 	\alpha \rangle}_{j_\ell}$.

Observe that there is a linear equation in the 
system $\mathcal{G}$ such that the coefficient 
$\lambda_{i(\gamma)}$ is non-trivial, for otherwise every automorphism 
in $\Aut(\mathcal{M}/P)$ fixing all coordinates except (possibly) 
the element $d_{i(\gamma)}$, which is the $i(\gamma)^\text{th}$-coordinate 
of the tuple $(\hat{a},\hat{b})$, 
must also fix  $c$, contradicting the minimality of $m$ and $n$. The set 
\[ B=\{ F(\alpha_{i_1}),\ldots, F(\alpha_{i_m}), 
F(\beta_{j_1}),\ldots, F(\beta_{j_n}) \} \]
consists of distinct elements, by the injectivity of $F$. Therefore, 
if suffices to show that the element $F(d_{i(\gamma)})$ is not algebraic 
over 
\[ A\cup \big(B\backslash \{ F(d_{i(\gamma)})  \}\big)  \]
to reach the desired contradiction. For this we need only show that the 
element $F(d_{i(\gamma)})$ is not contained in the set
\[
\{ \hat{\alpha},
{\langle \alpha \rangle}_{j_1},\ldots,{\langle \alpha 
	\rangle}_{j_n}  \}.
\]

If $d_{i(\gamma)}=\alpha_{i(\gamma)}$, we obtain the result
since
\begin{align*}
\Phi(F(d_{i(\gamma)}))&=\Phi(F(\alpha_\gamma)) \\
&=\Phi(\alpha_{\omega^{\gamma+1}}) \\
&=\omega^{\gamma+1}\geq\gamma+1 \\
&>\gamma= \Phi\big( 
(\hat{\alpha},
{\langle \alpha \rangle}_{j_1},\ldots,{\langle \alpha 
	\rangle}_{j_n})\big).
\end{align*}

Otherwise  $d_{i(\gamma)}=\beta_{i(\gamma)}$, so
\begin{align*}
\Phi(F(d_{i(\gamma)}))&=\Phi(F(\beta_{i(\gamma)}))\\
&=
\omega^{\max\left({i(\gamma)},\Phi({\langle \alpha 
			\rangle}_{i(\gamma)})\right)+1}+\omega^{{i(\gamma)}} \\ 
		&>
\omega^{ \Phi({\langle \alpha 
		\rangle}_{i(\gamma)}) +1  } =
\omega^{\gamma+1},
\end{align*} 
and we conclude the result analogous to the first case. 
\end{claimproof*}	
~\end{proof}

Whenever the sort $S$ is not $P$-internal, the additive cover does  
not eliminate imaginaries. The situation is different for  
finite imaginaries: We will see below that the additive cover 
$\mathcal{M}_0$ does not  eliminate finite imaginaries, however  the 
additive cover $\mathcal{M}_{1}$ does.

\begin{remark}\label{R:FiniteImagM0}
	Choose two generic independent elements $\alpha$ and $\beta$ be two in 
	the sort $P$. The finite subset 
	$\{(\alpha,0),(\beta,0)\}$ of $S$ has no real canonical parameter in 
	$\mathcal{M}_{0}$.
\end{remark}

\begin{proof}
	Assume that the tuple $e$ is a real canonical parameter of the set 
	$\{(\alpha, 
	0),(\beta, 0)\}$. Since the tuple $e$ is clearly definable over 
	$(\alpha,0), 
	(\beta, 0), P$, the projection $\pi(c)$ of every element $c$ in $S$ 
	appearing in $e$ 
	(if any) must be contained in the $\mathbb{Q}$-vector space generated 
	by 	$\alpha$ and $\beta$. 
	
	There is an automorphism $\tau$ of $P$ extending the 
	non-trivial permutation of the set $\{\alpha, \beta\}$, so it is easy 
	to show that there is a rational number $q$ such that 
	$\pi(c)=q\cdot(\alpha+\beta)$. Hence, the tuple $e$ is definable over 
	$(\alpha+\beta,0),P$.

	Therefore, any additive map $F$ with $F(\alpha)=1$ and 
	$F(\beta)=-1$ induces an automorphism $\sigma_F$ fixing $e$, yet it 
	does not permute $\{(\alpha,0),(\beta,0)\}$.
~\end{proof}

In order to show that the additive cover 
$\mathcal{M}_1$ eliminates finite imaginaries, 
we first provide a sufficient condition. 

\begin{proposition}\label{P:generalimag}
  An additive cover $\mathcal{M}$ eliminates 
 finite imaginaries, whenever every finite subset of $S$ on which $\pi$ is 
 injective has a real 
	canonical parameter.
\end{proposition}
\begin{proof}
	Let $A$ be the finite set $\{\bar{a}_1,\ldots,\bar{a}_n\}$ of 
	real $m$-tuples. 
	Every function 
	$\Phi:\{1,\ldots, m\}\longrightarrow\{P,S\}$ determines a 
	subset $A_{\phi}$ of $A$, according to whether the 
	$j^\text{th}$-coordinate lies in $P$ or $S$.  Every automorphism 
	permuting $A$ permutes 	each $A_{\Phi}$, so we may assume that for 
	every tuple in $A$, the coordinates have the same configuration (given 
	by the function $\Phi$).

	Since the canonical parameter is only determined up to  
	interdefinability, we may further assume (after possibly permuting the 
	coordinates) that there is a natural number  $0\leq k\leq m$ such that 
	for each tuple $\bar {a}_i$ in $A$:
	\begin{itemize}
		\item The $j^\text{th}$-coordinate $a_{i}^{j}$ lies in $S$ for 
		$1\leq j\leq k$.
		\item The $\ell^\text{th}$-coordinate $a_{i}^{\ell}$ lies in $P$ 
		for $k< \ell\leq m$.
	\end{itemize}
	For every coordinate $1\leq j\leq k$ set $A^j=\{a_{i}^{j} \ | \ 1\leq 
	i\leq n \}$ and $d_{i}^{j}$ 
	 the average of the subset $A^j \cap \pi^{-1}(\pi(a_{i}^{j}))$.  For 
	 $1\leq i\leq 
	 n$ let
now	$\varepsilon_{i}^{j}$ be the unique element in $P$ with
	$a_{i}^{j}=\varepsilon_{i}^{j} 
	\star d_{i}^{j}$. Consider the tuples 
	$\varepsilon_{i}=(\varepsilon_{i}^{1},\ldots,\varepsilon_{i}^{k})$
	and 
	\[ 
	\alpha_i=(\pi(a_{i}^{1}),\ldots,\pi(a_{i}^{k}),a_{i}^{k+1},\ldots,a_{i}^{m})\]
	 
	 in $P$. We need only show that the tuple
		\[ e=\big( \ulcorner \{(\varepsilon_{1},\alpha_{1}),\ldots,
	(\varepsilon_{n},\alpha_{n})
	\}\urcorner,\ulcorner 
	\{d_{1}^{1},\ldots,d_{n}^{1}\}\urcorner,\ldots,\ulcorner 
	\{d_{1}^{k},\ldots,d_{n}^{k}\}\urcorner \big) \]
	is a  canonical parameter of $A$. Note that $e$ is a real tuple since 
	the sets 
	$\{d_{1}^{j},\ldots,d_{n}^{j}\}$ have real canonical parameters, by 
	our assumption. 
	
	Let $\sigma$ be an automorphism. If $\sigma$ permutes the set $A$,  
	Lemma \ref{L:mean} yields
	that $\sigma$ permutes each set
	$\{d_{1}^{j},\ldots,d_{n}^{j}\}$ since the image of $A^j \cap 
	\pi^{-1}(\pi(a_{i}^{j}))$ under $\sigma$ is $A^j \cap 
	\pi^{-1}(\pi(a_{i^*}^{j}))$ for some index $i(\sigma)$ with 
	$\sigma(a_{i}^{j})=a_{i(\sigma)}^{j}$ and 
	$\sigma(\alpha_i)=\sigma(\alpha_{i(\sigma)})$. Therefore 
	$\sigma(\varepsilon_{i})=\varepsilon_{i(\sigma)}$, since 
	 $\sigma(d_{i}^{j})=d_{i(\sigma)}^{j}$. Hence $\sigma$ fixes $e$.  
	
Assume now that $\sigma$ fixes the tuple $e$. The tuple $\alpha_i$ is 
mapped to $\alpha_{i(\sigma)}$ and  \[ 
\sigma(a_{i}^{j})=\sigma(\varepsilon_{i}^{j}) \star 
\sigma(d_{i}^{j})=\varepsilon_{i(\sigma)}^{j}\star \sigma(d_{i}^{j}). \]
It suffices to show that  $\sigma(d_{i}^{j})=d_{i(\sigma)}^{j}$ to 
conclude that $\sigma$ permutes $A$. This follows immediately from
	\[\pi(\sigma(d_{i}^{j}))=\sigma(\pi(d_{i}^{j}))=
	\sigma(\alpha_{i}^{j})=\alpha_{i(\sigma)}^{j},\]  since $\sigma$ 
	permutes the set $\{d_{1}^{j},\ldots,d_{n}^{j}\}$. 
~\end{proof}

Thus, we will deduce that the additive cover $\mathcal{M}_1$ eliminates 
finite imaginaries, by applying Proposition \ref{P:generalimag}, lifting 
the corresponding canonical parameters of finite subsets of $P$ to $S$  
using the ring operations. 

\begin{corollary}\label{C:finiteImag}
	The additive cover $\mathcal{M}_{1}$ eliminates finite imaginaries.
\end{corollary}
\begin{proof}
	By Proposition \ref{P:generalimag}, we need only show that every 
	finite 
	subset  
	$\{a_1,\ldots,a_n\}$ of $S$, with pairwise distinct projections 
	$\pi(a_i)=\alpha_i$, has a real canonical parameter. 
	
	  For $1\leq i\leq n$ lift the $i^\text{\,th}$-symmetric function to 
	  $S$:  
	\begin{equation*}
	\tag{$\spadesuit$}
	b_i=\sum_{1\leq j_1<\dots<j_i\leq n}a_{j_1}\otimes\cdots\otimes 
	a_{j_i}. 
	\end{equation*}
     We 
	claim that the tuple $b=(b_1,\ldots,b_n)$ is a canonical parameter of 
	the set 
	$A=\{a_1,\ldots,a_n\}$. If the automorphism
	$\sigma$ permutes $A$, then it clearly fixes $b$. 
	Assume now that $\sigma$ fixes the tuple $b$. Write each element $a_i$ 
	of $A$ as $a_i=(\alpha_i,a_{i}')$, and similarly  
	$b_i=(\beta_i,b_{i}')$.
	 In the full structure 
	$(\mathbb{C},\mathbb{C}\times\mathbb{C})$ the definable condition 
	($\spadesuit$) uniquely translates into 
	\[ \beta_i=\sum_{1\leq j_1<\dots<j_i\leq 
		n}\alpha_{j_1}\cdots\alpha_{j_i}\]
	and the system of linear equations:
	\begin{equation*}
		\begin{pmatrix}
			1 & 1 & \cdots & 1 \\
			\sum\limits_{j\neq 1}\alpha_j & \sum\limits_{j\neq 2}\alpha_j 
			& \cdots & \sum\limits_{j\neq n}\alpha_j \\
			\sum\limits_{\substack{j_1<j_2 \\ j_1,j_2\neq 1}}\alpha_{j_1} 
			\alpha_{j_2} 
			& \sum\limits_{\substack{j_1<j_2 \\ j_1,j_2\neq 
			2}}\alpha_{j_1} \alpha_{j_2} 
			& \cdots 
			& \sum\limits_{\substack{j_1<j_2 \\ j_1,j_2\neq 
			n}}\alpha_{j_1} \alpha_{j_2}  \\
			\vdots & \vdots & \ddots & \vdots \\
			\prod\limits_{j\neq 1}\alpha_j & \prod\limits_{j\neq 
			2}\alpha_j & \cdots & \prod\limits_{j\neq n}\alpha_j
		\end{pmatrix}
		\begin{pmatrix}
		a_1' \\ a_2' \\ \vdots \\ \vdots \\ \vdots \\ a_n'
		\end{pmatrix}
		=
			\begin{pmatrix}
		b_1' \\ b_2' \\ \vdots \\ \vdots \\ \vdots \\ b_n'
		\end{pmatrix}
	\end{equation*}
	Since the tuple $(\beta_1,\ldots,\beta_n)$ encodes the finite set
	$\{\alpha_1,\ldots,\alpha_n\}$ and the above 
	matrix has determinant $\prod_{i<j}(\alpha_i-\alpha_j)\neq 0$, 
	we 
	conclude that the automorphism $\sigma$ permutes the set $A$.	
~\end{proof}

\section{The CBP in additive covers}\label{S:CBP_addcovers}

	As already stated in Remark \ref{R: Cex}, the CBP does not hold in the 
	additive cover $\mathcal{M}_1$ (see Example 
	\ref{E:covers}). For the sake of completeness, we will 
	now sketch a proof, using the
	terminology introduced so far. For generic independent 
	elements $a$, $b$ and $c$ in $S$, set $d=(a \otimes c) \oplus b$. 
	Assuming the CBP, the type
 	$\stp(a/c,d)$ is almost $P$-internal, since $\Cb(c,d/a,b)=(a,b)$. As 
 	the  	elements $a$, $c$ and $d$ are again (generic) 
 	independent, we conclude that the type $\stp(a)$ is almost 
 	$P$-internal,  contradicting the fact that $S$ is not 
 	almost 	$P$-internal.

The above  is a lifting to the sort $S$ of a configuration 
witnessing that the field $P$ is not one-based. We will now present 
another proof that the additive cover $\mathcal{M}_1$ does not have 
the CBP, using the so called group 
version of the CBP, which already appeared in \cite[Theorem 4.1]{KP06}

\begin{fact}\label{F:groupCBP}\textup{(}\cite[Fact 1.3]{HPP13}\textup{)}
	Let $G$ be a definable group in a theory with the CBP. Whenever $a$ is 
	in $G$ and the type $p=\tp(a/A)$ has finite stabilizer, then $p$ is 
	almost internal to the family of all non-locally modular minimal types.
\end{fact}

The failure of the group version of the CBP is another example of such a
lifting approach: Consider two generic independent elements $a$ and $b$ of 
$S$, and set $c=a\otimes b$. It is easy to see that $\stp(a,b,c)$ 
has trivial stabilizer, so the above Fact \ref{F:groupCBP} yields, 
assuming the CBP, that $S$ is almost $P$-internal, which is a 
contradiction. 
 
Now we will see that the additive cover $\mathcal{M}_1$ is already 
determined by its automorphism group over $P$. 

\begin{proposition}\label{P:M1AutVersion}
	If $\mathcal{M}$ is an additive cover such that 
	$\Aut(\mathcal{M}/P)$ corresponds to the group of derivations on 
	$\mathbb{C}$, then the product 
	\[(\alpha,a')\otimes (\beta,b')=(\alpha\beta,\alpha b' + \beta a')\]
	is definable in $\mathcal{M}$.
\end{proposition}
\begin{proof}
	Choose two generic independent elements 
	$\alpha$ and $\beta$ in $P$ and consider the elements 
	$a=(\alpha,0),b=(\beta,0)$ and $c=(\alpha\beta,0)$ 
	in $S$. The type $\tp(a,b,c/\alpha,\beta,\alpha\beta)$ is 
	$P$-internal and stationary, by Lemma \ref{L:stat}. 
	Since every element in its Galois group corresponds to a derivation, 
	we deduce 
	that for all elements
	\[
	\tilde a=(\alpha,a'), \tilde b=(\beta,b') \ \ \textnormal{ and } \ \ 
	\tilde c=(\alpha\beta,c')
	\]
	in $S$, we have that $a,b,c \equiv_{P} \tilde a,\tilde b,\tilde c$ if 
	and only 
	if $c'=\alpha 
	b'+\beta a'$.
	Therefore $c$ is definable over $a,b,P$.
	In fact, we obtain that $c$ is definable over $a, b$: 
	Let $\bar{\gamma}$ be a tuple of elements in $P$ such that $c$ is 
	definable over $a,b,\bar{\gamma}$.
	Now let $\bar{\varepsilon}$ be a maximal subtuple of $\bar{\gamma}$ 
	such that 
	\[\bar{\varepsilon} \ind_{\alpha,\beta}a,b,c.\]
	Note that $\bar{\gamma}\backslash\bar{\varepsilon}$ is algebraic over 
	$\bar{\varepsilon},a,b,c$. Therefore Remark $\ref{R:algFib}$
	implies that 
	$\bar{\gamma}\backslash\bar{\varepsilon}$ is algebraic over 
	$\bar{\varepsilon},\alpha,\beta$. 
	Hence $c$ is definable over $a,b,\bar{\varepsilon}$ and so,
	by independence, we deduce that $c$ is algebraic over $a,b$. 
	The average $(\alpha\beta,e')$ of the finite set consisting of the 
	$\{a,b\}$-conjugates of $c$ is 
	definable over $a,b$. Similarly as in the proof of Lemma \ref{L:stat}, 
	we deduce that 
	$e'$ is definable over $\alpha,\beta$. Hence, 
	\[ c=(-e')\star (\alpha\beta,e') \] 
	is definable over $a,b$.
	
	Let $\varphi(x,y,z)$ be
	a formula such that $c$ is the unique realization of $\varphi(a,b,z)$.
	For two generic independent elements 
	$a_1=(\alpha_1,a_1')$ and $b_1=(\beta_1,b_1')$ in $S$, choose a 
	derivation $D$ with $D(\alpha_1)=-a_1'$ 
	and $D(\beta_1)=-b_1'$ and let $\sigma_D$ be the induced automorphism 
	in $\Aut(\mathcal{M}/P)$. Furthermore, take a
	field automorphism $\tau$ of $P$ with $\tau(\alpha_1)=\alpha$
	and $\tau(\beta_1)=\beta$ and let $\sigma_{\tau}$ be the induced 
	automorphism of the additive cover as in Remark \ref{R:algFib}.
	Since $\sigma_{\tau}(\sigma_{D}(a_1,b_1))=(a,b)$, we deduce that
$\mathcal{M}\models\varphi(a_1,b_1,c_1)$ if and only if 
	$c_1=a_1\otimes b_1$. 
	
	Now we show that the multiplication $\otimes$ is globally definable, 
	following 
	the field version in Marker and Pillay's work \cite[Fact 1.5]{MP90}. 
	Set 
	\[X=\{ a \ | \ \varphi(\varepsilon\star a,b,(\varepsilon \star 
	a)\otimes b) \text{ for generic } b \text{ independent from } a \text{ 
		and every } \varepsilon \text{ in }P \}.\]
	Note that $\pi(X)$ is cofinite and $\pi^{-1}[\pi(a)]$ is contained in 
	$X$ for every $a$ in $X$.
	Note that $a=b$ if and only if they define the same germ, that is 
	$a\otimes c=b\otimes c$ for generic $c$ independent from $a$ and $b$, 
	since generic elements have an inverse.
	Let the finite set $P\backslash \pi(X)=\{\gamma_1,\ldots,\gamma_k\}$. 
	For $1\leq i\leq k$ choose $\alpha_i$ and $\beta_i$ in $\pi(X)$ such 
	that $\gamma_i=\alpha_i\beta_i$.
	Using the elements $(\gamma_i,0),(\alpha_i,0),(\beta_i,0)$ as 
	parameters, we can uniformly identify every element in the fiber of 
	$\gamma_i$ 
	with the product of two elements in $X$, namely
	$(\gamma_i,c')=(\alpha_i,0)\otimes \big(\varepsilon\star 
	(\beta_i,0)\big)$, where 
	$\varepsilon$ is the unique element in $P$ such that 
	$(\varepsilon\alpha_i)\star(\gamma_i,0)=(\gamma_i,c')$. Now we can 
	define the multiplication $\otimes$ globally as the composition of 
	germs of 
	elements in $X$.
	~\end{proof} 
 
We will now show that the CBP holds in the additive cover $\mathcal{M}_0$ 
and more generally whenever there is 
essentially no additional structure on the sort $S$. 

\begin{proposition}\label{P:CBPpureCover}
	The CBP holds in an  additive cover $\mathcal{M}$, whenever every 
	additive map on $\mathbb{C}$ induces an automorphism in 
	$\Aut(\mathcal{M}/P)$. 
\end{proposition}
In particular, the additive cover $\mathcal{M}_0$ has the CBP.  
\begin{proof}
	Recall that we need only consider real types over models in order to 
	deduce that the CBP holds.
	Let $p(x)$ be the type of some finite real tuple $\bar{a}$ of length 
	$k$ 
	over an elementary substructure $N$. In order to show that
 the type 
	$\stp(\Cb(p)/\bar{a})$ is almost
	$P$-internal, choose a formula $\varphi(x;\bar{b},\bar{\gamma})$ in 
	$p$ of least Morley rank and Morley degree one, where $\bar{b}$ is a 
	tuple of elements in $S\cap N$ and $\bar{\gamma}$ is a tuple of 
	elements 
	in 
	$P \cap N$. 
	
	We claim that every automorphism in 
	$\Aut(\mathcal{M}/P,\bar{a})$ fixes the canonical base $\Cb(p)$, which 
	is (interdefinable with) the canonical parameter $\ulcorner 
	\textnormal{d}_{p} x 
	\varphi(x;y)\urcorner$. For 	this, it suffices to show that every 
	such automorphism sends the tuple 
	$\bar{b}$ to another realization of the formula 
	$\textnormal{d}_{p} x \varphi(x; y_1, \gamma)$.
	
	Write $\bar{a}=(a_1,\ldots,a_k)$ and \[ \alpha_i=\begin{cases}
\pi(a_i),  \text{ if $a_i$ is in $S$}\\ 
a_i \text{ otherwise.} 
	\end{cases} \] For  
	$\bar{b}=(b_1,\ldots,b_n)$, set 
	$\beta_i=\pi(b_i)$. 
  We may assume (after possibly reordering) that 
	$(\beta_1,\ldots,\beta_m)$ is a maximal subtuple of $\bar{\beta}$ 
	which 
	is $\mathbb{Q}$-linearly independent over $\bar{\alpha}$. So,
	\[ \beta_j = \sum\limits_{i=1}^{m}q_i\cdot\beta_i + 
	\sum\limits_{i=1}^{k}r_i\cdot\alpha_i \]
	for $m+1\leq j\leq n$ and rational numbers $q_i$ and $r_i$. In order 
	to show that $\bar{b}$ is mapped by the automorphism 
	$\sigma$ of $\Aut(\mathcal{M}/P,\bar{a})$ to another realization of 
	the formula 
	$\textnormal{d}_{p} x \varphi(x;y_1, \gamma)$, it suffices to show 
	that 
	\[ N\models \forall \varepsilon_1,\ldots,\varepsilon_m \in P\ \ 
	\textnormal{d}_{p} x \varphi(x;\bar{\varepsilon}\star\bar{b},\gamma)   
	\]
	 where $\bar{\varepsilon}=(\varepsilon_1,\ldots,\varepsilon_n)$ with
	 \[ \varepsilon_j =  \sum\limits_{i=1}^{m}q_i\cdot\varepsilon_i  \]
	 for $m+1\leq j\leq n$. Indeed: since $N$ is an elementary 
	 substructure of $\mathcal{M}$, the above implies that 	 \[ 
	 \mathcal{M}\models \forall \varepsilon_1,\ldots,\varepsilon_m \in 
	 P\ \  \textnormal{d}_{p} x 
	 \varphi(x,\bar{\varepsilon}\star\bar{b},\bar{\gamma}), \] so 
	 $\sigma(\bar b)= F_\sigma(\bar b)\star \bar b$ realizes  
	 $\textnormal{d}_{p} x \varphi(x;\bar y_1,\gamma)$, as desired.

	 So, let $\varepsilon_1,\ldots,\varepsilon_m$ be in $P\cap N$ and set
	 $\varepsilon_j =  \sum_{i=1}^{m}q_i\cdot\varepsilon_i 
	 $ for $m+1\leq j\leq n$. Choose an additive 
	 map $G$ vanishing on $\alpha_i$ for $1\leq i\leq k$ and with
	 $F(\beta_i)=\varepsilon_i$ for $1\leq i \leq m$. Hence \[G(\beta_j)= 
	 \sum\limits_{i=1}^{m}q_i\cdot\varepsilon_i , \]
	 so  the image of $\bar b$ under the automorphism $\sigma_G$ induced 
	 by $F$ lies in $N$. 
	 Hence $\sigma_{G}(\bar{b})=\bar \varepsilon \star \bar b$ realizes  
	 $\textnormal{d}_{p} x \varphi(x,y_1,\gamma)$ since 
	 $\sigma_{G}(\bar{a})=\bar{a}$, as desired. 
~\end{proof}

\begin{remark}\label{R:rsCBP}
	The above proof shows that the canonical base of a real stationary 
	type $\stp(a/B)$ is definable over $a,P$ which is stronger than
	$P$-internality.
	As we will see below this does not hold for all 
	imaginary types. 
\end{remark}

Palac\'in and Pillay ~\cite{PP17} considered a strengthening of the CBP, 
called the \textit{strong canonical base property}, which we reformulate 
in the setting of additive covers: Given a (possibly imaginary) type 
$p=\stp(a/B)$, 
its canonical base $\Cb(p)$ is algebraic over $a, \bar d$, where 
$\stp(\bar d)$ is $P$-internal. If we denote by $\mathcal{Q}$ the 
family 
types over $\acleq(\emptyset)$ which are $P$-internal,  then the 
strong CBP holds if and only if  every Galois group $G$ relative to 
$\mathcal{Q}$ is \emph{rigid} \cite[Theorem 3.4]{PP17}, that is, the 
connected component of every 
definable subgroup of $G$ is definable over $\acl(\ulcorner G\urcorner)$.

Notice that no additive cover where the sort $S$ is not almost 
$P$-internal can have the strong CBP: 
For the two generic independent elements $a=(\alpha,0)$ and $b=(\beta,0)$ 
in $S$, the stationary 
$P$-internal type $\tp(a,b/\alpha,\beta)$ is fundamental and has Galois 
group 
$(\mathbb{C}^2,+)$. This is clearly a $\mathcal{Q}$-internal type whose 
Galois group $G$ (relative to $\mathcal{Q}$) is a definable subgroup of 
$(\mathbb{C}^2,+)$. Since vector groups are never rigid, it suffices to 
show that $G=\mathbb{C}^2$ (compare to \cite[Proposition 4.9]{JJP20}). 
Otherwise, 
the 
element 
$b$ is algebraic over $a,\bar{d}$, where $\stp(\bar d)$ belongs to 
$\mathcal Q$ (up to permutation of $a$ and $b$). Hence, the type 
$\stp(b/a)$, and thus $S$, is almost $P$-internal.

The question whether a  Galois-theoretic interpretation
of the CBP exists arose in \cite{PP17}. We conclude this section by 
showing that 
no \textit{pure} Galois-theoretic  account of the CBP can be provided.
We already noticed in Remark \ref{R:Gal} that, whenever the sort $S$ in an 
additive cover is not almost $P$-internal, then the Galois groups relative 
to $P$ are precisely all definable subgroups of $(\mathbb{C}^n,+)$, as $n$ 
varies. In particular, all such additive covers share the same Galois 
groups 
(relative to $P$). We will now see that the same holds for the Galois 
groups relative to $\mathcal{Q}$.

\begin{lemma}\label{L:GalAccount}
	All additive covers where the sort $S$ is 
	not almost $P$-internal share  
	the same Galois groups relative to $\mathcal{Q}$. 
\end{lemma}
\begin{proof}
	Note that $\mathcal{Q}$-internality coincides with $P$-internality. 
	Moreover, 
	the Galois group relative to $\mathcal{Q}$ is a subgroup of the Galois 
	group relative to $P$, which by Remark \ref{R:Gal} is a definable 
	subgroup of some $(\mathbb{C}^n,+)$. So it suffices to show that 
	every definable subgroup $G$ of $(\mathbb{C}^n,+)$ 
	appears as a Galois group relative to $\mathcal{Q}$.
	 
	Choose a tuple $\bar{a}$ of elements 
	$a_1=(\alpha_1,0),\ldots,a_n=(\alpha_n,0)$ in the sort 
	$S$ with generic independent elements $\alpha_i$ in $P$ and
	set
	\[ E = \{ \bar{x}\in S^n \ |\ \exists \bar{g}\in G \bigwedge_{i=1}^{n} 
	g_i \star a_i = x_i    \}.  \]
	The proof of Remark \ref{R:Gal} shows that the stationary type 
	$\stp(\bar a/\ulcorner E\urcorner)$ 
	is $P$-internal and fundamental with Galois group $G$. Moreover, for 
	every set $B$ of parameters we have that  \[\stp(\bar a/\ulcorner 
	E\urcorner, B) \vdash 
	\tp(\bar a/\ulcorner E\urcorner,B,P).\]

	We now show that the Galois group $H$ relative to $\mathcal{Q}$ equals 
	$G$. Assume for a contradiction that $H$ is a 
	proper subgroup of $G$. The group $G$ (and $H$ relative to $G$) is 
	given by a system of linear equations in echelon form, so we find an 
	index $1\le k\le n$ and a tuple $\bar d$ with $\stp(\bar d)$  
	$P$-internal such that the element $a_k$ is not algebraic over
		$\bar{a}_{>k},\ulcorner E\urcorner$, yet it is algebraic over 
		$\bar{a}_{>k},\ulcorner 
		E\urcorner,\bar d$.
		
	By $\mathcal{P}$-internality of $\stp(\bar d)$, there is a set of 
	parameters $C$ with $C \ind \bar{d},\bar{a},\ulcorner E\urcorner$ such 
	that $\bar d$ is 
	definable over $C,P$.
	The above yields that $a_k$ is algebraic over  $\bar{a}_{>k},\ulcorner 
	E\urcorner,C,P$ and therefore over $\bar{a}_{>k},\ulcorner 
	E\urcorner$, which yields the desired contradiction.
\end{proof}

\section{Preservation of internality in additive 
covers}\label{S:PI_addcovers}

In this section we will show that the additive cover $\mathcal{M}_1$ does 
not preserve 
internality on intersections nor internality on quotients. We will start 
with the latter, whose proof is considerably simpler.

\begin{proposition}\label{P: M1PropB}
	The additive cover $\mathcal{M}_1$ does not preserve internality on 
	quotients.
\end{proposition}
\begin{proof}
	Choose generic independent elements $a,b$ and $c$ in $S$ and set 
	$d=(a\otimes c)\oplus b$. 
	Consider now the following definable set:
	\begin{align*} 
	E &= \{ (x,y)\in S^2 \ |\ \pi(x)=\pi(a) \ \& \   \pi(y)=\pi(b) \ \& \ 
	d=(x\otimes c)\oplus y \} 
	\end{align*}
	Since  the canonical parameter
	$\ulcorner E\urcorner$ is clearly definable over $c,d,\pi(a),\pi(b)$ 
	and the type $\stp(c,d,\pi(a),\pi(b)/\pi(c),\pi(d))$
	is $P$-internal, 
	we deduce that the type \[ \stp(\ulcorner E\urcorner/\pi(c),\pi(d))\] 
	is $P$-internal. 
	\begin{claim*}
		The type $\stp(\ulcorner E\urcorner / \pi(a),\pi(b))$ is 
		$P$-internal.
	\end{claim*}
	\begin{claimproof*}
		Choose elements $a_1$ and $b_1$ in the fiber of $\pi(a)$, resp. 
		$\pi(b)$, such that
		\[ a_1,b_1 \ind_{\pi(a),\pi(b)} \ulcorner E\urcorner.  \]
		Note that every automorphism $\sigma$ in $\Aut(\mathcal{M}_1/P)$ 
		fixing the elements $a_1$ and $b_1$ must fix 
		$\pi^{-1}(\pi(a))\times\pi^{-1}(\pi(b))$, so $\sigma$
		permutes $E$. In particular, the canonical parameter $ \ulcorner 
		E\urcorner$ is definable over $a_1,b_1,P$, as desired.
	\end{claimproof*}
	We assume now that $\mathcal M_1$ preserves internality on 
	quotients in order to 
	reach a 
	contradiction. Since 
	\[\acleq\big(\pi(a),\pi(b)\big)\cap\acleq\big(\pi(c),\pi(d)\big)=\acleq(\emptyset),\]
	we deduce
	that the type $\stp(\ulcorner E\urcorner)$
	is almost $P$-internal.
	Therefore there is a real subset $C$ of $S$ with
	$C \ind \ulcorner E\urcorner$ such that the canonical parameter 
	$\ulcorner E\urcorner$ is algebraic over $C,P$. Note that in 
	particular 
	\[\pi(C),\pi(a) \ind \pi(b).\]
	Choose now a derivation $D$ vanishing both on $\pi(C)$ and on 
	$\pi(a)$ with
	$D(\pi(b))=1$. The induced automorphism $\sigma_D$ fixes $C$ and $P$ 
	pointwise but $\ulcorner E\urcorner$ has an infinite orbit, yielding 
	the desired contradiction.
	~\end{proof}
\begin{remark}
	The previous set is definable in every additive cover, since $E$ equals
	\[
	 \{ (x,y)\in S^2 \ |\ \exists\big(\lambda,\mu)\in P^2 ( \lambda\star 
		 a = x  \ \& \  \mu\star b= y \ \& \  \lambda \cdot \pi(c)+\mu=0 
		 \big)  \}. 
	\] 
	The main cause for the failure of preservation of internality on 
	quotients is that $E$ is definable over $c,d,P$ in $\mathcal{M}_1$.
\end{remark}

\begin{proposition}\label{P: M1PropA}
	 The additive cover $\mathcal{M}_1$ does not preserve internality on 
	 intersections.
\end{proposition}
\begin{proof}
	Choose generic independent elements $a_1$ and $a_2$ in $S$ and 
	$\varepsilon$ in $P$ generic over $a_1, a_2$. Set 
	$\bar\alpha=(\alpha_1,\alpha_2)=(\pi(a_1),\pi(a_2))$. 
	Consider now the definable set
	\[ E= \{ (x,y)\in S^2 \ |\ \exists(\lambda,\mu)\in P^2 ( \lambda\star 
	a=x \ \& \ 
	\mu\star b=y \ \& \ \varepsilon\cdot\lambda+\mu=0 ) \}.  \]
	Choose $\beta_1$ in $P$ generic over $
	\ulcorner 
	E\urcorner, \bar\alpha,\varepsilon$ as well as elements
	$\beta_2$ and $\beta_3$ in $P$ with 
	\begin{align}
	0&=\beta_1 \alpha_1 + \frac{1}{2}\beta_2 \alpha_{1}^2 + 
	\frac{1}{3}\beta_3 \alpha_{1}^3+\alpha_2 \\
	0&=\beta_1 +\beta_2 \alpha_1 + \beta_3 \alpha_{1}^2 - \varepsilon
	\end{align}
	This is possible because the matrix
	\[ \begin{pmatrix}
		\frac{\alpha_1^2}{2} & \frac{\alpha_{1}^3}{3} \\
		\alpha_1 & \alpha_{1}^2
	\end{pmatrix}\]
	has determinant $\frac{\alpha_{1}^4}{2}-\frac{\alpha_{1}^4}{3}\neq 0$.
	Since $\beta_2$ and $\beta_3$ are definable over 
	$\beta_1,\bar\alpha,\varepsilon$, we get the independence
	\begin{equation*}
	\tag{$\blacklozenge$}	
	\bar\beta \ind_{\bar\alpha,\varepsilon} 
	\ulcorner E\urcorner,
	\end{equation*}
	where $\bar\beta=(\beta_1,\beta_2,\beta_3)$.
\begin{claim}
The type $\stp(\ulcorner E\urcorner/\bar\beta)$ 
	is $P$-internal.
\end{claim}
	\begin{claimproof}
		Let $b_1,b_2$ and $b_3$ be elements in $S$ such that $b_i$ is in 
		the 
		fiber of $\beta_i$ with 
		\[ b_1,b_2,b_3 \ind_{\bar\beta} \ulcorner 
		E\urcorner, \bar\alpha,\varepsilon \]
		We show that every automorphism $\sigma$ in 
		$\Aut(\mathcal{M}_{1}/P)$ fixing $b_1,b_2$ and $b_3$ must permute 
		$E$. 
		Recall that 
		$F_\sigma$ is the derivation on $\mathbb{C}$ induced by the 
		automorphism $\sigma$. 
		Since $F_{\sigma}(\beta_i)=0$, we deduce from equations (1) and 
		(2) 
		that $\varepsilon\cdot 
		F_{\sigma}(\alpha_1)+F_{\sigma}(\alpha_2)=0$. Hence, the 
		automorphism $\sigma$ permutes the set $E$.
	\end{claimproof}

\begin{claim} The intersection
$\acleq(\ulcorner 
	E\urcorner)\cap\acleq(\bar\beta)=\acleq(\emptyset)$.
\end{claim}
	\begin{claimproof}
		Because of the independence $(\blacklozenge)$, we need only show 
		that 
		\[\acleq(\bar\beta)\cap\acleq
		(\bar\alpha,\varepsilon)=\acleq(\emptyset).\]

	Choose tuples $\bar{\beta}',\bar{\alpha}',\varepsilon',
	\bar{\beta}'',\bar{\alpha}'',\varepsilon'',
	\bar{\beta}'''$ such that
	\[ 
	\bar{\beta},\bar{\alpha},\varepsilon \equiv 
	\bar{\beta}',\bar{\alpha},\varepsilon \equiv
	\bar{\beta}',\bar{\alpha}',\varepsilon' \equiv
	\bar{\beta}'',\bar{\alpha}',\varepsilon' \equiv
	\bar{\beta}'',\bar{\alpha}'',\varepsilon'' \equiv
	\bar{\beta}''',\bar{\alpha}'',\varepsilon''
	\]
	with 
	\begin{align*}
	\bar{\beta}' \ind_{\bar{\alpha},\varepsilon} \bar{\beta} \qquad
	\bar{\alpha}',\varepsilon' \ind_{\bar{\beta}'} 
	\bar{\beta},\bar{\alpha},\varepsilon \qquad
	\bar{\beta}'' \ind_{\bar{\alpha}',\varepsilon'} 
	\bar{\beta},\bar{\alpha},\varepsilon,\bar{\beta}' \qquad
	\bar{\alpha}'',\varepsilon'' \ind_{\bar{\beta}''} 
	\bar{\beta},\bar{\alpha},\varepsilon,\bar{\beta}',\bar{\alpha}',\varepsilon'
	\end{align*} and 
	\begin{align*}
		\bar{\beta}''' \ind_{\bar{\alpha}'',\varepsilon''} 
		\bar{\beta},\bar{\alpha},\varepsilon,\bar{\beta}',\bar{\alpha}',\varepsilon',
		\bar{\beta}''.
	\end{align*}
	Since
		 \[ \acleq(\bar{\beta})\cap\acleq(\bar{\alpha},\varepsilon)\subset 
		 \acleq(\bar{\beta})\cap\acleq(\bar{\beta}'''),  \]
		 we need only show the independence $\bar{\beta} \ind \bar{ 
		 \beta''' } 
		 $. Note first that the whole configuration has Morley rank 9:
		 \[
		 \RM(\bar{\beta},\bar{\alpha},\varepsilon,
		 \bar{\beta}',\bar{\alpha}',\varepsilon',
		 \bar{\beta}'',\bar{\alpha}'',\varepsilon'',
		 \bar{\beta}''') = 
		 \RM(\beta_1,\alpha_1,\alpha_2,\varepsilon,
		 \beta_1',\alpha_1',\beta_1'',\alpha_1'',\beta_1''') =9.
		 \]
		 Since 	 \begin{multline*}
		  \RM(\bar{\beta}''',\bar{\beta},\alpha_1,\alpha_{1}',\alpha_{1}'')=
		    \\ \RM( 
		    \bar{\beta}'''/\bar{\beta},\alpha_1,\alpha_{1}',\alpha_{1}'') +
		      \RM(\alpha_{1}'' / \bar{\beta},\alpha_1,\alpha_{1}'  ) + 
		      \RM( \alpha_{1}' / \bar{\beta},\alpha_1 ) + \\ +
		      \RM(\alpha_1 /\bar{\beta}) + \RM(\bar{\beta}) = 
		      \RM( 
		      \bar{\beta}'''/\bar{\beta},\alpha_1,\alpha_{1}',\alpha_{1}'')
		       + 6,
		 \end{multline*}
		it suffices to show that 
		$\alpha_{2},\varepsilon,\bar{\beta}',\alpha_{2}',\varepsilon',
		\bar{\beta}'',\alpha_{2}''$ and $\varepsilon''$ are all algebraic 
		over the tuple 
		$(\bar{\beta}''',\bar{\beta},\alpha_1,\alpha_{1}',\alpha_{1}'')$.
			 Clearly 
		 $\alpha_2,\varepsilon,\alpha_{2}''$ and $\varepsilon''$ are 
		 algebraic over
		 $\bar{\beta}''',\bar{\beta},\alpha_1,\alpha_{1}''$.  Furthermore 
		 we 
		 have the following system of linear equations:
		  \begin{align*}
		 \begin{pmatrix}
		 6\alpha_1 & 3\alpha_{1}^2 & 2\alpha_{1}^3 & 0 & 0 & 0 & 0 & 0 \\
		 1 & \alpha_1 & \alpha_{1}^2 & 0 & 0 & 0 & 0 & 0 \\
		 6\alpha_{1}' & 3\alpha_{1}'^2 & 2\alpha_{1}'^3 & 6 & 0 & 0 & 0 & 
		 0 \\
		 1 & \alpha_{1}' & \alpha_{1}'^2 & 0 & 1 & 0 & 0 & 0 \\
		 0 & 0 & 0 & 6 & 0 & 6\alpha_{1}' & 3\alpha_{1}'^2 & 
		 2\alpha_{1}'^3 \\
		 0 & 0 & 0 & 0 & 1 & 1 & \alpha_{1}' & \alpha_{1}'^2 \\
		 0 & 0 & 0 & 0 & 0 & 6\alpha_{1}'' & 3\alpha_{1}''^2 & 
		 2\alpha_{1}''^3 \\
		 0 & 0 & 0 & 0 & 0 & 1 & \alpha_{1}'' & \alpha_{1}''^2 \\
		 \end{pmatrix}
		 \begin{pmatrix}
		 \beta_{1}' \\ \beta_{2}' \\ \beta_{3}' \\ \alpha_{2}' \\ 
		 \varepsilon' \\
		 \beta_{1}'' \\ \beta_{2}'' \\ \beta_{3}'' 
		 \end{pmatrix}
		 =
		 \begin{pmatrix}
		 -6\alpha_2 \\
		 \varepsilon \\
		 0 \\
		 0 \\
		 0 \\
		 0 \\
		 -6\alpha_{2}'' \\
		 \varepsilon'' \\
		 \end{pmatrix}
		 \end{align*}
Thus, we need only show that the above matrix has non-zero 
		 determinant
		 \begin{align*}
		 &6 
		 \begin{vmatrix}
		 6\alpha_1 & 3\alpha_{1}^2 & 2\alpha_{1}^3  \\
		 1 & \alpha_1 & \alpha_{1}^2 \\
		 1 & \alpha_{1}' & \alpha_{1}'^2
		 \end{vmatrix}
		 \begin{vmatrix}
		 6\alpha_{1}' & 3\alpha_{1}'^2 & 2\alpha_{1}'^3 \\
		 6\alpha_{1}'' & 3\alpha_{1}''^2 & 2\alpha_{1}''^3 \\
		 1 & \alpha_{1}'' & \alpha_{1}''^2
		 \end{vmatrix} 
		 -6
		 \begin{vmatrix}
		 6\alpha_1 & 3\alpha_{1}^2 & 2\alpha_{1}^3 \\
		 1 & \alpha_1 & \alpha_{1}^2 \\
		 6\alpha_{1}' & 3\alpha_{1}'^2 & 2\alpha_{1}'^3 
		 \end{vmatrix}
		 \begin{vmatrix}
		 1 & \alpha_{1}' & \alpha_{1}'^2 \\
		 6\alpha_{1}'' & 3\alpha_{1}''^2 & 2\alpha_{1}''^3 \\
		 1 & \alpha_{1}'' & \alpha_{1}''^2 
		 \end{vmatrix} \\
		 &=72 \alpha_{1}^2 \alpha_{1}'^2 \alpha_{1}''^2 
		 (\alpha_{1}-\alpha_{1}')(\alpha_{1}-\alpha_{1}'')(\alpha_{1}''-\alpha_{1}')\neq
		  0.
		 \end{align*} 
	~\end{claimproof}
		If $\mathcal M_1$ had preservation of internality on 
		intersections, then 
		the type
	\[\stp(\ulcorner E\urcorner / 
	\acleq(\ulcorner E\urcorner)\cap\acleq(\beta_1,\beta_2,\beta_3))\] 
	would be almost $P$-internal, by Claim 1, and so would be 
	$\stp(\ulcorner E\urcorner)$, by the previous claim, which yields a 
	contradiction, exactly as in the proof of Proposition \ref{P: M1PropB}.
~\end{proof}

Recall that an additive cover  preserves internality on 
intersections, resp. on quotients, if and only if every almost 
$P$-internal 
type is good, resp. special, by Propositions \ref{P:propA} and 
\ref{P:propB}. For real types, the property of being special follows 
directly from almost internality.

\begin{remark}\label{R:realpart}
Almost $P$-internal 
	real types are special in every additive cover. 
\end{remark}
\begin{proof}
    We may assume that the sort $S$ is not almost $P$-internal. By a 
    straight-forward forking calculation (cf. \cite[Theorem 2.5]{zC12} or 
    Proposition \ref{P:propB}), it suffices to show that, whenever the 
    real type $\stp(a/B)$ is almost $P$-internal, with $a$ a single 
    element in $S$, then $\alpha=\pi(a)$ is algebraic over $B$. 
    
    Choose a set of parameters $B_1$  
	with $B_1 \ind_{B} a$ and $a$ algebraic over $B_1,P$. We need only 
	show that $\alpha$ is algebraic over $B_1$. Otherwise, choose an 
	element $a_1$ of $S$ in the fiber of $\alpha$ generic over $B_1$. The 
	elements $a$ and $a_1$ are interdefinable over $P$, so $a_1$ is 
	algebraic over $B_1, P$, contradicting that $S$ is not almost 
	$P$-internal.
\end{proof}
 Propositions 
\ref{P: M1PropA} and \ref{P: M1PropB} and the  above remark give a 
negative answer to Question~\ref{Q:2}.

\begin{corollary}\label{C:AB_imag}
	There is a stable theory of finite Morley rank, where every stationary 
	real almost 
	$\p$-internal type is special, yet internality on intersections is not 
	preserved.
\end{corollary}

We can now conclude this work relating the failure of the CBP and 
elimination of finite imaginaries, always in the context of additive 
covers.  For this, we need the following easy remark, which follows 
immediately from 
\cite[Remark 1.1 (2)]{zC12}.
\begin{remark}\label{R:intersec} Given tuples $a$ and $b$ in an ambient 
model of an $\omega$-stable theory
such that  $\RM(a)-\RM(a/b)=1$ and $b=\Cb(a/b)$, the intersection
	\[\acleq(a)\cap 
	\acleq(b)=\acleq(\emptyset).\]
\end{remark}
\begin{theorem}\label{T:finImagCBP}
	Suppose that the sort $S$ in the additive cover $\mathcal{M}$ is not 
	almost $P$-internal. If $\mathcal{M}$ eliminates finite imaginaries, 
	then it cannot preserve internality on quotients, so in particular 
	the CBP  does not hold.
\end{theorem}
In a  forthcoming work, we will explore whether the converse holds. We 
believe that similar techniques  show that  an additive cover as above 
cannot even preserve internality on intersections, but we have not yet 
pursued this problem thoroughly.
\begin{proof}
	We assume that the additive cover $\mathcal{M}$ eliminates finite 
	imaginaries and that the sort $S$ is not almost $P$-internal. In order 
	to show the failure of preserving internality on quotients, we will 
	find a similar configuration 
	to $(a\otimes c)\oplus b=d$, resonating with Martin's work \cite{gM88} 
	on 
	recovering 
	multiplication. 
	
	Choose two generic independent elements 
	$a_0=(\alpha_0,0)$ and 
	$a_1=(\alpha_1,0)$ in $S$. The
	real canonical parameter of the finite set 
	$\{ a_0,a_1 \}$ is not definable over $a_0 \oplus a_1, P$: Indeed, 
	since $S$ 
	is not almost $P$-internal, there is an automorphism $\sigma$ in 
	$\Aut(\mathcal{M}/P)$ with $\sigma(a_0)=1\star a_0$ and 
	$\sigma(a_1)=(-1)\star 
	a_1$, so $\sigma(a_0 \oplus a_1)=a_0\oplus a_1$, but $\sigma$ does not 
	permute 
	$\{ a_0,a_1 \}$. Choose now some coordinate $e$ of the real canonical 
	parameter which 
	is not definable over $a_0\oplus a_1,P$. 
	Note that $\varepsilon=\pi(e)$ is definable over $\alpha_0,\alpha_1$, 
	by 
	Remark \ref{R:algFib}.  Therefore $\varepsilon=r(\alpha_0,\alpha_1)$ 
	for 
	some symmetric 
	rational function $r(X,Y)$ over $\mathbb{Q}$. Let $\rho(x,y,z)$ be a 
	formula such that $e$ is the unique element realizing
	$\rho(a_0,a_1,z)$.
	
	We now proceed according to whether $r(\alpha_0,Y)$ is a 
	polynomial 
	map. Assume first that the map $r_{\alpha_0}(Y)=r(\alpha_0,Y)$ is not 
	polynomial.
	 
	As in the proof of \cite[Lemma 3.2]{gM88}, there are natural numbers 
	$n_1,\ldots,n_k$ 
	such that the 
	degree of the numerator $P_{\alpha_0}(Y)$ of the rational function
	\[
	\sum_{j=0}^{k}(-1)^{k-j}
	\sum_{1\leq i_1<\dots<i_j\leq k} 
	r_{\alpha_0}\big(Y+n_{i_1}+\cdots+n_{i_j}\big) 
	\] 
	is strictly smaller  than the degree of its denominator 
	$Q_{\alpha_0}(Y)$. 
	 
	 For $1\leq i_1<\dots<i_j\leq k$, the formula $\rho(a_0,a_1\oplus 
	 (n_{i_1}+\cdots+n_{i_j},0),z)$
	has a unique realization $e_{i_1,\dots,i_j}$, since 
	\[\alpha_1 \equiv_{\alpha_0} \alpha_1+n_{i_1}+\cdots+n_{i_j},\]
	so by Remark \ref{R:algFib} 
	\[a_1 \equiv_{a_0} a_1\oplus(n_{i_1}+\cdots+n_{i_j},0). \] 
 	
		Set now 
	\[e_j=\sum_{1\leq i_1<\dots<i_j\leq k} e_{i_1,\dots,i_j}\]
	and \begin{multline*} 
	\psi(x, y,z)= 
	\exists \bar{z} \Big(
	\bigwedge_{j=0}^{k} \ \ 
	\bigwedge_{1\leq i_1<\dots<i_j\leq k}
	\rho(x,y\oplus(n_{i_1}+\cdots+n_{i_j},0),z_{i_1,\dots,i_j}) \ \ 
	\land \\
	z=
	\sum_{j=0}^{k}(-1)^{k-j}
	\sum_{1\leq i_1<\dots<i_j\leq k}
	z_{i_1,\dots,i_j}
	\Big).
	\end{multline*}
	Note that the element
	\[ \sum_{j=0}^{k}(-1)^{k-j} e_j\]
	is the unique realization of $\psi(a_0,a_1,z)$ and its projection to 
	$P$ is
	\[ \frac{P_{\alpha_0}(\alpha_1)}{Q_{\alpha_0}(\alpha_1)}.\]
	
	By Remark \ref{R:algFib}, every element in the fiber of $\alpha_0$ has 
	the same type as $a_0$ over $a_1,P$, so the formula
	\[ \forall u \Big( \pi(u)=x \rightarrow  \Big(
	\exists ! z  \psi(u, a_1,z) \land \forall w \Big( \psi(u, a_1,w) 
	\rightarrow \pi(w)= \frac{P_{x}(\alpha_1)}{Q_{x}(\alpha_1)} 
	\Big)\Big)\Big)\]
\noindent belongs to the generic type $\tp(\alpha_0/a_1)$ in $P$. 
Therefore, there exists an algebraic number $\xi$ realizing it such that 
$\deg(Q_\xi(Y))>\deg(P_\xi(Y))$. Write now $\varphi(y,z)=\psi( (\xi, 0), 
y,z)$ and choose generic independent elements $a,b$ and $c$ in $S$ with 
	projections
	\[\pi(a)=\alpha,\pi(b)=\beta \ \ \textnormal{ and } \ \ \pi(c)=\gamma.
	\]
	The formula $\varphi$ will play the role of the multiplication 
	$\otimes$, so let
	$d=(\delta,d')$ be the unique element such that 
	\[ \mathcal{M}\models \exists z \big(\varphi(a\oplus c,z)\land z\oplus 
	b=d\big). \]
	\begin{claim}
		The intersection 
		$\acleq(\alpha,\beta)\cap\acleq(\gamma,\delta)=\acleq(\emptyset)$.
	\end{claim}
	\begin{claimproof}
			Since 
			$\RM(\alpha,\beta)-\RM(\alpha,\beta/\gamma,\delta)=2-1=1$, it 
		suffices to show by Remark \ref{R:intersec} that 
		$\Cb(\alpha,\beta/\gamma,\delta)$ is interdefinable with 
		$(\gamma,\delta)$. 
	
		Choose elements $\alpha'$ and $\beta'$ such that 
		\[\alpha',\beta' \equiv_{\gamma,\delta} \alpha,\beta \ \ 
		\textnormal{ 
			and } \ \ \alpha',\beta' \ind_{\gamma,\delta} \alpha,\beta \ 
			,\]
		so
		\[
		\frac{P_\xi(\alpha+\gamma)}{Q_\xi(\alpha+\gamma)}+\beta=\delta=
		\frac{P_\xi(\alpha'+\gamma)}{Q_\xi(\alpha'+\gamma)}+\beta'. 
		\]
		Therefore
		\[
		P_\xi(\alpha+\gamma)Q_\xi(\alpha'+\gamma)-P_\xi(\alpha'+\gamma)Q_\xi(\alpha+\gamma)+
		(\beta-\beta')Q_\xi(\alpha+\gamma)Q_\xi(\alpha'+\gamma)=0.
		\]
		Since 
		\[\deg(Q_\xi(Y))>\deg(P_\xi(Y)),\]
		we need only show $\beta\neq\beta'$, for then 
		$\gamma$ is algebraic over $\alpha,\beta,\alpha',\beta'$ and 
		hence so is $\delta$, as desired. 
		
		We assume for a contradiction that 
		$\beta=\beta'$. Hence $\beta$ is algebraic over $\gamma,\delta$, 
		so  the equation
		\[ P_\xi(\alpha+\gamma)=(\delta-\beta)Q_\xi(\alpha+\gamma)\]
		yields that $\alpha$ is also algebraic over $\gamma,\delta$, which 
		is a blatant 
		contradiction.
	\end{claimproof}
	As in Proposition \ref{P: M1PropB}, with the definable set
	\begin{align*} 
		E &= \big\{ (x,y)\in S^2 \ |\ \pi(x)=\alpha \ \& \   \pi(y)=\beta 
		\ \& \ 
		\exists z \big(\varphi(x\oplus c,z)\land z\oplus y=d\big) \big\} 
	\end{align*}
	we can easily prove that 
	the types 
	\[ 
	\stp(\ulcorner E\urcorner/\gamma,\delta) \ \ \textnormal{ and } \ \ 
	\stp(\ulcorner E\urcorner/\alpha,\beta)
	\] 
	are $P$-internal, since $(\xi,0)$ is internal over $\acl(\emptyset)$.

	We assume now that $\mathcal M$ preserves internality on 
	quotients in order to reach a  contradiction. By 
	 the above claim, 
	the type $\stp(\ulcorner E\urcorner)$
	is almost $P$-internal. 
	Therefore, there is a set $C$ of parameters with
	$C \ind \ulcorner E\urcorner,a,b$ such that the canonical parameter 
	$\ulcorner E\urcorner$ is algebraic over $C,P$. 
	Note that in 
	particular 
	\[C,a \ind b.\]
	Since the sort $S$ is not almost $P$-internal, there is an 
	automorphism 
	$\sigma$ in $\Aut(\mathcal{M}/P)$ fixing $C$ and $a$, yet 
	$\sigma(b)\neq b$. 
	The orbit of $\ulcorner E\urcorner$ under $\sigma$ is hence infinite, 
	which gives the desired contradiction.
	
	The remaining case is that the rational function 
	$r(\alpha_0,Y)$ is 
	polynomial. For a natural number $m$,  write $r(X,mX+Y)$ as 
	\[  
	r(X,mX+Y)=\sum_{i=0}^{n} \frac{P_{m,i}(X)}{Q_{m,i}(X)}Y^i,
	\]
	with coprime polynomials $P_{m,i}(X)$ and $Q_{m,i}(X)$ over 
	$\mathbb{Q}$ with
	$P_{m,n}\neq 
	0$ (for $r$ is not the zero map).
	
	\begin{claim}
		There exists a natural number $m$ such that 
		$\deg(P_{m,i})\neq \deg(Q_{m,i})$
		for some $i>0$.
	\end{claim}
	\begin{claimproof}
		Note that $n>0$ because $r(X,Y)$ is symmetric and non-constant.
		We may assume that $\deg(P_{0,i})= \deg(Q_{0,i})$ for all 
		$i>0$, since otherwise we are done.
		If $n>1$, then
		\begin{align*}
		\frac{P_{1,n-1}(X)}{Q_{1,n-1}(X)}&=
		\frac{P_{0,n}(X)}{Q_{0,n}(X)} X + 
		\frac{P_{0,n-1}(X)}{Q_{0,n-1}(X)}\\
		&=\frac{P_{0,n}(X)Q_{0,n-1}(X)X + P_{0,n-1}(X)Q_{0,n}(X)  
		}{Q_{0,n}(X) Q_{0,n-1}(X)}
		\end{align*}
		implies
		\[   
		\deg(P_{1,n-1})= \deg(Q_{1,n-1}) + 1,
		\]
		so the claim follows. Thus, we are left with the case 
		$n=1$, where
		\begin{align*}
		r(X,Y)=\frac{P_{0,1}(X)}{Q_{0,1}(X)} Y + 
		\frac{P_{0,0}(X)}{Q_{0,0}(X)} 
		= \frac{P_{0,1}(X)Q_{0,0}(X)Y + P_{0,0}(X)Q_{0,1}(X)  
		}{Q_{0,1}(X) Q_{0,0}(X)}.
		\end{align*}
		The map 
		\begin{align*}
		r(\alpha_0,Y)=r(Y,\alpha_0)=\frac{P_{0,1}(Y)Q_{0,0}(Y)\alpha_0 + 
		P_{0,0}(Y)Q_{0,1}(Y)  
		}{Q_{0,1}(Y) Q_{0,0}(Y)}
		\end{align*}
		is polynomial	and since $\alpha_0 \equiv \alpha_0 + 1$, so is 
		the map
		\begin{align*}
		\frac{P_{0,1}(Y)Q_{0,0}(Y)(\alpha_0+1) + 
			P_{0,0}(Y)Q_{0,1}(Y)  
		}{Q_{0,1}(Y) Q_{0,0}(Y)}.
		\end{align*}
		Since $P_{0,1}$ and $Q_{0,1}$ as well as $P_{0,0}$ and $Q_{0,0}$
		are coprime,
		it follows that $Q_{0,0}=\lambda Q_{0,1}$ for some 
		rational number $\lambda\neq 0$. We deduce that both
		\[
		\frac{\lambda \alpha_0 P_{0,1}(X) + P_{0,0}(X)}{\lambda 
		Q_{0,1}(X)}	
		\] 
		and
		\[
		\frac{\lambda (\alpha_0 +1) P_{0,1}(X) + P_{0,0}(X)}{\lambda 
		Q_{0,1}(X)}	
		\] 
		are polynomials. Hence, every root $\zeta$ of $Q_{0,1}$ 
		is a root of 
		\[
		\lambda \alpha_0 P_{0,1} + P_{0,0}
		\ \ \textnormal { and of } \ \
		\lambda (\alpha_0 +1) P_{0,1} + P_{0,0}
		\]
		and therefore $P_{0,1}(\zeta)=0$. 
		This implies that $Q_{0,1}$ is constant, since
		$P_{0,1}$ and $Q_{0,1}$ 
		are coprime.
		It follows that $P_{0,1}$ cannot be constant, since otherwise
		the symmetric function 	$r(X,Y)$ would equal to  
		$q_1\cdot(X+Y)+q_0$ for 
		some rational numbers $q_1$ and $q_0$, which yields that
		the element $e$ would be definable over $a_0\oplus a_1, P$, a 
		contradiction.
	\end{claimproof}	
	Fix now a natural number $m$ as in the previous claim and choose as 
	before generic independent elements $a$, $b$ and $c$  in $S$ with 
	projections
	\[\pi(a)=\alpha,\pi(b)=\beta \ \ \textnormal{ and } \ \ \pi(c)=\gamma.
	\]
	Let
	$d=(\delta,d')$ be the unique element such that 
	\[ \mathcal{M}\models \exists z \big(\rho(a,(m\cdot a) \oplus 
	c,z)\land z\oplus 
	b=d\big). \]
	
	Considering the set
	\begin{align*} 
		\big\{ (x,y)\in S^2 \ |\ \pi(x)=\alpha \ \& \   \pi(y)=\beta \ \& 
		\ 
		\exists z \big(\rho(a,(m\cdot a) \oplus c,z)\land z\oplus y=d\big) 
		\big\}, 
	\end{align*} we need only show as before that 
	\[\acleq(\alpha,\beta)\cap\acleq(\gamma,\delta)=\acleq(\emptyset).\]
		The strategy is the same as in the proof of Claim 1. 	Choose 
		 elements $\alpha'$ and $\beta'$ such that 
		\[\alpha',\beta' \equiv_{\gamma,\delta} \alpha,\beta \ \ 
		\textnormal{ 
			and } \ \ \alpha',\beta' \ind_{\gamma,\delta} \alpha,\beta \ 
			.\]
		Note that
		\[  
		r(\alpha,m\cdot\alpha+\gamma)+\beta=\delta=r(\alpha',m\cdot\alpha'+\gamma)+\beta'
		 ,
		\]
	so 
		\[
		r(\alpha,m\cdot\alpha+\gamma)-r(\alpha',m\cdot\alpha'+\gamma)+\beta-\beta'=0.
		\]
		Now Claim 2 implies that $\gamma$ is algebraic over 
		$\alpha,\beta,\alpha',\beta'$, since $\alpha \ind \alpha'$ (for 
		otherwise both $\alpha$ and $\beta$ are algebraic over 
		$\gamma,\delta$). It 
		follows that  
		$\delta$ is also algebraic over 
		$\alpha,\beta,\alpha',\beta'$, as desired. 
	~\end{proof}

\end{document}